\documentclass[a4paper]{amsart}
\usepackage{amssymb}
\usepackage{verbatim}
\usepackage{pb-diagram}
\usepackage[dvipdfmx]{graphicx}
\theoremstyle{plain}
  \newtheorem{thm}{Theorem}[section]
  \newtheorem{lem}[thm]{Lemma}
  
  \newtheorem{prop}[thm]{Proposition}
  \newtheorem{conj}[thm]{Conjecture}

  \newtheorem*{obs*}{Observation}
\theoremstyle{definition}

\theoremstyle{remark}
  \newtheorem{rem}[thm]{Remark}

\newcommand{\Z}{\mathbb{Z}}
\newcommand{\C}{\mathbb{C}}
\newcommand{\R}{\mathbb{R}}
\newcommand{\Q}{\mathbb{Q}}
\newcommand{\Vol}{\operatorname{Vol}}
\newcommand{\CS}{\operatorname{CS}}
\newcommand{\Li}{\operatorname{Li}}

\newcommand{\cs}{\operatorname{cs}}
\newcommand{\SL}{\rm{SL}}
\newcommand{\Tr}{\operatorname{Tr}}

\newcommand{\arccosh}{\operatorname{arccosh}}
\newcommand{\Res}{\operatorname{Res}}

\renewcommand{\Re}{\operatorname{Re}}
\renewcommand{\Im}{\operatorname{Im}}

\numberwithin{equation}{section}
\allowdisplaybreaks
\begin{document}
\title[The colored Jones polynomial of the figure-eight knot]
{The colored Jones polynomial, the Chern--Simons invariant, and the Reidemeister torsion of the figure-eight knot}
\author{Hitoshi Murakami}
\address{
Department of Mathematics,
Tokyo Institute of Technology,
Oh-okayama, Meguro, Tokyo 152-8551, Japan
}
\email{starshea@tky3.3web.ne.jp}
\date{\today}
\begin{abstract}
We show that from the asymptotic behavior of an evaluation of the colored Jones polynomial of the figure-eight knot we can extract the Chern--Simons invariant and the twisted Reidemeister torsion associated with a representation of the fundamental group of the knot complement to the two-dimensional complex special linear group.
\end{abstract}
\dedicatory{Dedicated to the memory of my father, Akira Murakami (1930--2011)}
\keywords{colored Jones polynomial, volume conjecture, Chern--Simons invariant, Reidemeister torsion}
\subjclass[2000]{Primary 57M27 57M25 57M50}
\thanks{The authors are supported by Grant-in-Aid for Challenging Exploratory Research (21654053)}
\maketitle
%%%%%%%%%%%%%%%%%%%%%%%%%%%%%%%%%%%%%%%%%%%%%%%%%%%%%%%%%%%%%%%%%%%%
\section{Introduction}
Let $J_N(K;q)$ be the colored Jones polynomial of a knot $K$ in the three-sphere $S^3$ associated with the $N$-dimensional irreducible representation of $\mathfrak{sl}(2;\C)$.
We normalize it so that $J_N(\text{unknot};q)=1$.
Note that $J_2(K;q)$ is the celebrated Jones polynomial \cite{Jones:BULAM385} after a suitable change of variable.
\par
In 1995 Kashaev introduced a complex valued knot invariant for each natural number $N$ by using quantum dilogarithm \cite{Kashaev:MODPLA95} and observed that its asymptotic behavior for large $N$ determines the hyperbolic volume for several hyperbolic knots \cite{Kashaev:LETMP97}.
He also conjectured that it is also true for any hyperbolic knot.
Here a knot is called hyperbolic if its complement has a unique complete hyperbolic structure with finite volume.
It is proved in 2001 by J.~Murakami and the author that his invariant coincides with $J_N\bigl(K;\exp(2\pi\sqrt{-1}/N)\bigr)$ \cite{Murakami/Murakami:ACTAM12001}.
We also generalized Kashaev's conjecture to the following volume conjecture.
\begin{conj}[Volume Conjecture \cite{Kashaev:LETMP97,Murakami/Murakami:ACTAM12001}]
Let $K$ be a knot in $S^3$ and $\Vol(K)$ denote the simplicial volume of $S^3\setminus{K}$.
Then the following equality would hold:
\begin{equation}\label{eq:VC}
  \lim_{N\to\infty}
  \frac{\log\left|J_N\bigl(K;\exp(2\pi\sqrt{-1}/N)\bigr)\right|}{N}
  =
  \frac{\Vol(K)}{2\pi}.
\end{equation}
\end{conj}
See for example \cite{Murakami:Columbia} about recent developments of the conjecture and its generalizations.
\par
As one of the  generalizations Yokota and the author \cite{Murakami/Yokota:JREIA2007} proved that for the figure-eight knot the colored Jones polynomial knows much more.
Actually we showed that if we perturb the parameter $2\pi\sqrt{-1}$ a little the corresponding limit determines the $\SL(2;\C)$ Chern--Simons invariant associated with an irreducible representation of $\pi_1(S^3\setminus{K})$ to $\SL(2;\C)$ in the sense of Kirk and Klassen \cite{Kirk/Klassen:COMMP93}.
In fact we showed the following theorem.
\begin{thm}[\cite{Murakami/Yokota:JREIA2007}]\label{thm:Murakami/Yokota}
Let $E$ be the figure-eight knot.
There exists a neighborhood $U\subset\C$ of $0$ such that if $u\in(U\setminus{\pi\sqrt{-1}\Q})\cup\{0\}$ then the following limit exists:
\begin{equation*}
  \lim_{N\to\infty}
  \frac{\log J_N\bigl(K;\exp((u+2\pi\sqrt{-1}/N))\bigr)}{N}.
\end{equation*}
Moreover the limit determines the $\SL(2;\C)$ Chern--Simons invariant associated with an irreducible representation of $\pi_1(S^3\setminus{E})$ to $\SL(2;\C)$ which is determined by the parameter $u$.
\end{thm}
On the other hand, Andersen and Hansen \cite[Theorem~1]{Andersen/Hansen:JKNOT2006} refined the volume conjecture for the figure-eight knot as follows.
\begin{thm}[\cite{Andersen/Hansen:JKNOT2006}]\label{thm:Andersen/Hansen}
The following asymptotic equivalence holds:
\begin{equation*}
\begin{split}
  &J_N\bigl(E;\exp(2\pi\sqrt{-1}/N)\bigr)
  \\
  \underset{N\to\infty}{\sim}&
  \frac{1}{3^{1/4}}
  N^{3/2}
  \exp\left(\frac{N\Vol(E)}{2\pi}\right)
  \\
  =\hspace{3mm}&
  2\pi^{3/2}
  \left(\frac{2}{\sqrt{-3}}\right)^{1/2}
  \left(\frac{N}{2\pi\sqrt{-1}}\right)^{3/2}
  \exp\left(\frac{N}{2\pi\sqrt{-1}}\times\sqrt{-1}\Vol(E)\right).
\end{split}
\end{equation*}
Note that the twisted Reidemeister torsion and the Chern--Simons invariant associated with the unique complete hyperbolic structure of $S^3\setminus{E}$ are $2/\sqrt{-3}$ and $\sqrt{-1}\Vol(E)$ respectively.
\end{thm}
Note that we write $f(N)\underset{N\to\infty}{\sim}g(N)$ if and only if $\lim_{N\to\infty}f(N)/g(N)=1$ and that \eqref{eq:VC} follows from the equivalence relation above when $K$ is the figure-eight knot.
\par
In this paper we refine Theorem~\ref{thm:Murakami/Yokota} as Theorem~\ref{thm:Andersen/Hansen} for the case where $u$ is real.
\begin{thm}\label{thm:main}
Let $u$ be a real number with $0<u<\log((3+\sqrt{5})/2)=0.9624\dots$ and put $\xi:=2\pi\sqrt{-1}+u$.
Then we have the following asymptotic equivalence of the colored Jones polynomial of the figure-eight knot $E$:
\begin{equation}\label{eq:main}
  J_N(E;\exp(\xi/N))
  \\
  \underset{N\to\infty}{\sim}
  \frac{\sqrt{-\pi}}{2\sinh(u/2)}
  T(u)^{1/2}
  \left(\frac{N}{\xi}\right)^{1/2}
  \exp\left(\frac{N}{\xi}S(u)\right),
\end{equation}
where
\begin{equation*}
  S(u)
  :=
  \Li_2(e^{u-\varphi(u)})-\Li_2(e^{u+\varphi(u)})-u\varphi(u)
\end{equation*}
and
\begin{equation*}
  T(u)
  :=
  \frac{2}{\sqrt{(e^u+e^{-u}+1)(e^u+e^{-u}-3)}}.
\end{equation*}
Here $\varphi(u):=\arccosh(\cosh(u)-1/2)$ and 
\begin{equation*}
  \Li_2(z)
  :=
  -\int_{0}^{z}\frac{\log(1-x)}{x}\,dx
\end{equation*}
is the dilogarithm function.
\end{thm}
Note that $S(u)$ defines the $\SL(2;\C)$ Chern--Simons invariant and $T(u)$ is the cohomological twisted Reidemeister torsion, both of which are associated with an irreducible representation of $\pi_1(S^3\setminus{E})$ into $SL(2;\C)$ sending the meridian to an element with eigenvalues $\exp(u/2)$ and $\exp(-u/2)$.
See Section~\ref{sec:interpretation} for details.
\begin{rem}
Since the figure-eight knot is amphicheiral, we have $J_N(E;q^{-1})=J_N(E;q)$.
Thus if $u<0$ we have
\begin{equation*}
\begin{split}
  J_N\bigl(E;\exp((u+2\pi\sqrt{-1})/N)\bigr)
  &=
  J_N\bigl(E;\exp((-u-2\pi\sqrt{-1})/N)\bigr)
  \\
  &=
  \overline{J_N\bigl(E;\exp((-u+2\pi\sqrt{-1})/N)\bigr)},
\end{split}
\end{equation*}
where $\overline{z}$ denotes the complex conjugate of $z$.
So if we prove Theorem~\ref{thm:main} for $u>0$, we have a similar asymptotic equivalence for $u<0$.
Details are left to the readers.
\end{rem}
\par
Theorem~\ref{thm:main} confirms the following conjecture in the case of the figure-eight knot for real $u$ with $0<u<\log((3+\sqrt{5})/2)$.
\begin{conj}[\cite{Gukov/Murakami:FIC2008,Dimofte/Gukov:Columbia}]
Let $K$ be a hyperbolic knot.
Then there exists a neighborhood $U\in\C$ of $0$ such that if $u\in U\setminus\pi\sqrt{-1}\Q$, we have the following asymptotic equivalence:
\begin{equation*}
  J_N(K;\exp(\xi/N))
  \\
  \underset{N\to\infty}{\sim}
  \frac{\sqrt{-\pi}}{2\sinh(u/2)}
  T(K;u)^{1/2}
  \left(\frac{N}{\xi}\right)^{1/2}
  \exp\left(\frac{N}{\xi}S(K;u)\right),
\end{equation*}
where $\xi:=2\pi\sqrt{-1}+u$, $T(K;u)$ is the cohomological twisted Reidemeister torsion and $S(K;u)$ is the $\SL(2;\C)$ Chern--Simons invariant, both of which are associated with an irreducible representation of $\pi_1(S^3\setminus{K})$ into $SL(2;\C)$ sending the meridian to an element with eigenvalues $\exp(u/2)$ and $\exp(-u/2)$.
\end{conj}
For physical interpretations of this conjecture, see for example \cite{Gukov:COMMP2005,Dimofte/Gukov:Columbia}.
\par
For torus knots we know the following result.
Let $T(a,b)$ be the torus knot of type $(a,b)$ for positive coprime integers $a$ and $b$.
It is known that the $\SL(2;\C)$ character variety of $\pi_1(S^3\setminus{T(a,b)})$ has $(a-1)(b-1)/2$ components \cite{Klassen:TRAAM1991} (see also \cite{Munoz:REVMC2009}).
Such components are indexed by a positive integer $k$ that is not a multiple of $a$ or $b$.
See \cite[\S~2]{Hikami/Murakami:Bonn} for details.
Let $\rho_k$ be an irreducible representation in the component indexed by $k$, 
$S_k(u)$ be the Chern--Simons invariant associated with $\rho_k$ with $\exp(\pm u/2)$ the eigenvalues of the image of the meridian by $\rho_k$, and $T_k$ be the cohomological twisted Reidemeister torsion associated with $\rho_k$.
Then we have the following formulas \cite{Hikami/Murakami:Bonn}:
\begin{align*}
  S_k(u)
  &:=
  \frac{-\bigl(2k\pi\sqrt{-1}-ab(2\pi\sqrt{-1}+u)\bigr)^2}{4ab}
  \\
  \intertext{and}
  T_k
  &:=
  \frac{16\sin^2(k\pi/a)\sin^2(k\pi/b)}{ab}.
\end{align*}
Dubois and Kashaev \cite{Dubois/Kashaev:MATHA2007}, and Hikami and the author \cite{Hikami/Murakami:Bonn} obtain the following asymptotic equivalences.
\begin{thm}[\cite{Dubois/Kashaev:MATHA2007}]
For $u=0$ we have
\begin{multline*}
  J_N\bigl(T(a,b);\exp(2\pi\sqrt{-1}/N)\bigr)
  \\
  \underset{N\to\infty}{\sim}
  \frac{\pi^{3/2}}{2ab}
  \left(\frac{N}{2\pi\sqrt{-1}}\right)^{3/2}
  \sum_{k=1}^{ab-1}
  (-1)^{k+1}k^2
  T_k^{1/2}
  \exp\left(\frac{N}{\xi}S_k(0)\right).
\end{multline*}
Note that since $T_k$ vanishes if $a$ or $b$ divides $k$, the summation is for all the irreducible components of the character variety.
\end{thm}
\begin{thm}[\cite{Hikami/Murakami:Bonn}]
Let $u$ be a complex number with $0<|u|<2\pi/(ab)$.
Then we have
\begin{equation*}
  J_N\bigl(T(a,b);\exp(\xi/N)\bigr)
  \underset{N\to\infty}{\sim}
  \frac{1}{\Delta(T(a,b);e^{u})}
\end{equation*}
when $\Re{u}>0$ and
\begin{multline*}
  J_N\bigl(T(a,b);\exp(\xi/N)\bigr)
  \\
  \underset{N\to\infty}{\sim}
  \frac{1}{\Delta(T(a,b);e^u)}
  +
  \frac{
  \sqrt{-\pi}}{2\sinh(u/2)}
  \sum_{k=1}^{ab-1}
  (-1)^{k}
  T_k^{1/2}
  \left(\frac{N}{\xi}\right)^{1/2}
  \exp\left(\frac{N}{\xi}S_k(u)\right)
\end{multline*}
when $\Re{u}<0$, where $\xi:=u+2\pi\sqrt{-1}$ and $\Delta(T(a,b);t)$ is the Alexander polynomial.
\par
See \cite{Hikami/Murakami:Bonn} for more details.
\end{thm}
\par
The paper is organized as follows.
In Section~\ref{sec:integral} we give an integral formula for the colored Jones polynomial using the quantum dilogarithm.
We study its asymptotic behavior by using the saddle point method to give a proof of Theorem~\ref{thm:main} in Section~\ref{sec:approximation}.
In Section~\ref{sec:interpretation} we give topological interpretations of $S(u)$ and $T(u)$.
Sections~\ref{sec:S_gamma} to \ref{sec:Phi_0} are devoted to miscellaneous calculations.
%%%%%%%%%%%%%%%%%%%%%%%%%%%%%%%%%%%%%%%%%%%%%%%%%%%%%%%%%%%%%%%%%%%%%
\section{Integral formula for the colored Jones polynomial}
\label{sec:integral}
In this section we use the quantum dilogarithm function to express the colored Jones polynomial of the figure-eight knot as an integral.
We mainly follow \cite{Andersen/Hansen:JKNOT2006}.
\par
First of all we recall the following formula due to Habiro and Le (see for example \cite{Masbaum:ALGGT12003}).
\begin{equation}\label{eq:Habiro_Le}
\begin{split}
  J_N(E;q)
  &=
  \sum_{k=0}^{N-1}
  \prod_{l=1}^{k}
  \left(q^{(N-l)/2}-q^{-(N-l)/2}\right)
  \left(q^{(N+l)/2}-q^{-(N+l)/2}\right)
  \\
  &=
  \sum_{k=0}^{N-1}
  q^{-kN}
  \prod_{l=1}^{k}
  \left(1-q^{N-l}\right)
  \left(1-q^{N+l}\right).
\end{split}
\end{equation}
\par
For a complex number $\gamma$ with $\Re(\gamma)>0$, define the quantum dilogarithm $S_{\gamma}(z)$ as follows \cite{Faddeev:LETMP1995}:
\begin{equation*}
  S_{\gamma}(z)
  :=
  \exp
  \left(
    \frac{1}{4}
    \int_{C_R}
    \frac{e^{zt}}{\sinh(\pi t)\sinh(\gamma t)}
    \frac{dt}{t}
  \right),
\end{equation*}
where $|\Re(z)|<\pi+\Re(\gamma)$ and $C_R$ is $(-\infty,-R]\cup\Omega_R\cup[R,\infty)$ with $\Omega_R:=\{R\exp(\sqrt{-1}(\pi-s))\mid0\le s\le\pi\}$ for $0<R<\min\{\pi/|\gamma|,1\}$.
Note that the poles of the integrand are $0,\pm\sqrt{-1},\pm2\sqrt{-1},\dots$ and $\pm\pi\sqrt{-1}/\gamma,\pm2\pi\sqrt{-1}/\gamma,\dots$.
\begin{rem}
Note that in \cite{Andersen/Hansen:JKNOT2006} it is assumed that $\gamma$ is real and $0<\gamma<1$ but we can define $S_{\gamma}(z)$ when $\Re(\gamma)>0$.
See \cite[(3.21)]{Dimofte/Gukov/Lenells/Zagier:CNTP2010}.
(Our quantum dilogarithm $S_{\gamma}(z)$ is equal to $\Phi(z/(2\pi);\gamma/\pi)$ in \cite{Dimofte/Gukov/Lenells/Zagier:CNTP2010}.)
We give a proof of the analyticity of $S_{\gamma}$ in Lemma~\ref{lem:analyticity}.
\end{rem}
The following formula is well known and its proof can be found in \cite[p.~530]{Andersen/Hansen:JKNOT2006}.
Note that they assume that $\gamma$ is real but their proof is also valid in our case.
\begin{lem}
If $|\Re(z)|<\pi$, then we have
\begin{equation}\label{eq:S_gamma}
  \left(1+e^{\sqrt{-1}z}\right)S_{\gamma}(z+\gamma)
  =
  S_{\gamma}(z-\gamma).
\end{equation}
\end{lem}
\begin{comment}
Then the set of the poles of (the analytic continuation of) $S_{\gamma}$ is $\{(2k+1)\gamma+\pi(2l+1)\mid k\in\Z_{\ge0},l\in\Z_{\ge0}\}$ and the set of zeroes is $\{-(2k+1)\gamma+\pi(2l+1)\mid k\in\Z_{\ge0},l\in\Z_{\ge0}\}$.
\end{comment}
Putting $\gamma:=(2\pi-\sqrt{-1}u)/(2N)$ and $z:=\pi-\sqrt{-1}u-2l\gamma$ ($l=1,2,\dots,N-1$) in \eqref{eq:S_gamma}, we have
\begin{equation*}
  \left(1+e^{\sqrt{-1}(\pi-\sqrt{-1}u-2l\gamma)}\right)
  S_{\gamma}(\pi-\sqrt{-1}u-2l\gamma+\gamma)
  =
  S_{\gamma}(\pi-\sqrt{-1}u-2l\gamma-\gamma).
\end{equation*}
So for $k=1,2,\dots,N-1$ we have
\begin{equation*}
\begin{split}
  \prod_{l=1}^{k}\bigl(1-\exp((N-l)\xi/N)\bigr)
  &=
  \prod_{l=1}^{k}
  \left(1+e^{\sqrt{-1}(\pi-\sqrt{-1}u-2\pi l/N+\sqrt{-1}ul/N)}\right)
  \\
  &=
  \prod_{l=1}^{k}
  \frac{S_{\gamma}(\pi-\sqrt{-1}u-(2l+1)\gamma)}
       {S_{\gamma}(\pi-\sqrt{-1}u-(2l-1)\gamma)}
  \\
  &=
  \frac{S_{\gamma}(\pi-\sqrt{-1}u-(2k+1)\gamma)}
       {S_{\gamma}(\pi-\sqrt{-1}u-\gamma)}.
\end{split}
\end{equation*}
Similarly we have
\begin{equation*}
  \prod_{l=1}^{k}\bigl(1-\exp((N+l)\xi/N)\bigr)
  =
  \frac{S_{\gamma}(-\pi-\sqrt{-1}u+\gamma)}
       {S_{\gamma}(-\pi-\sqrt{-1}u+(2k+1)\gamma)}.
\end{equation*}
Therefore we have
\begin{multline*}
  J_N\bigl(E;\exp(\xi/N)\bigr)
  \\
  =
  \frac{S_{\gamma}(-\pi-\sqrt{-1}u+\gamma)}{S_{\gamma}(\pi-\sqrt{-1}u-\gamma)}
  \sum_{k=0}^{N-1}
  \exp(-ku)
  \frac{S_{\gamma}(\pi-\sqrt{-1}u-(2k+1)\gamma)}
       {S_{\gamma}(-\pi-\sqrt{-1}u+(2k+1)\gamma)}.
\end{multline*}
\par
Using $S_{\gamma}$ we define
\begin{equation}\label{eq:g_N_def}
  g_N(w)
  :=
  \exp(-Nuw)
  \frac{S_{\gamma}(\pi-\sqrt{-1}u+\sqrt{-1}\xi w)}
        {S_{\gamma}(-\pi-\sqrt{-1}u-\sqrt{-1}\xi w)}.
\end{equation}
Since $S_{\gamma}(z)$ is defined for $|\Re(z)|<\pi+\Re(\gamma)$, the function $g_N(w)$ is defined for $w$ with $|\pi-\Im(\xi w)|<\pi+\Re(\gamma)=\pi+\pi/N$, that is, $w$ is in the strip $-(2\pi/u)\Re(w)-\pi/(Nu)<\Im(w)<-(2\pi/u)\Re(w)+2\pi/u+\pi/(Nu)$ (Figure~\ref{fig:contour}).
\begin{figure}[h]
\includegraphics[scale=1]{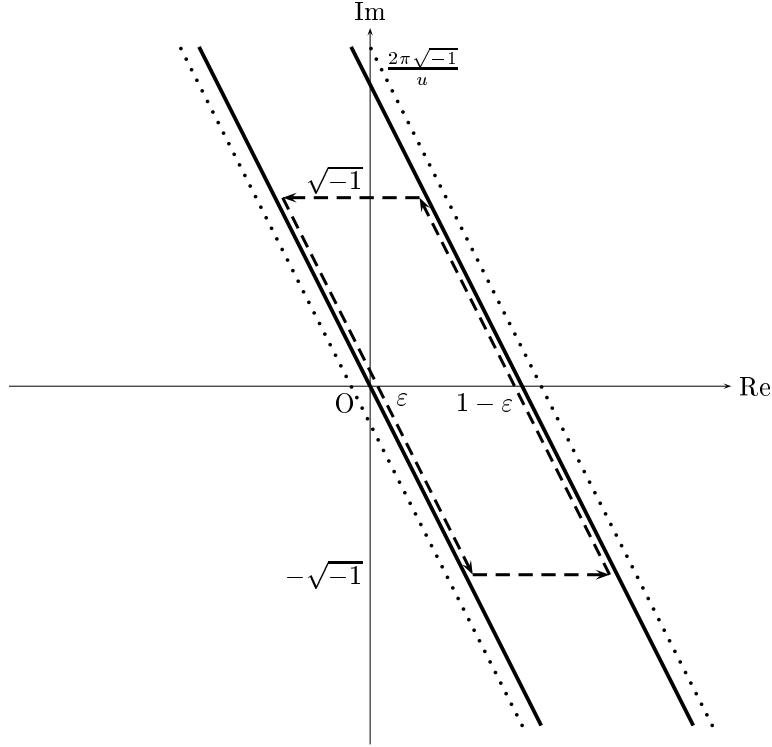}
\caption{The function $g_N$ ($\Phi$, respectively) is defined between the two dotted (thick, respectively) lines.
The dashed parallelogram indicates the contour $C_{+}(\varepsilon)\cup C_{-}(\varepsilon)$.}
\label{fig:contour}
\end{figure}
\par
For $0<\varepsilon<1/(4N)$, let $C_{+}(\varepsilon)$ be the polygonal line that connects $1-\varepsilon$, $1-u/(2\pi)-\varepsilon+\sqrt{-1}$, $-u/(2\pi)+\varepsilon+\sqrt{-1}$, and $\varepsilon$, and $C_{-}(\varepsilon)$ be the polygonal line that connects $\varepsilon$, $u/(2\pi)-\sqrt{-1}$, $1+u/(2\pi)-\sqrt{-1}$, and $1-\varepsilon$.
See Figure~\ref{fig:contour}.
Put $C(\varepsilon):=C_{-}(\varepsilon)\cup C_{+}(\varepsilon)$.
Note that the domain of $g_N(w)$ contains $C(\varepsilon)$.
\par
By the residue theorem we have
\begin{multline*}
  J_N\bigl(E;\exp(\xi/N)\bigr)
  \\
  =
  \frac{S_{\gamma}(-\pi-\sqrt{-1}u+\gamma)}{S_{\gamma}(\pi-\sqrt{-1}u-\gamma)}
  \frac{\sqrt{-1}\exp(u/2)N}{2}
  \int_{C(\varepsilon)}
  \tan(N\pi w)
  g_N(w)
  dw
\end{multline*}
since the set of the poles of $\tan(N\pi w)$ inside $C(\varepsilon)$ is $\{(2k+1)/(2N)\mid k=0,1,2,\dots,N-1\}$ and the residue of each pole is $-1/(N\pi)$.
\par
Putting
\begin{equation*}
  G_{\pm}(N,\varepsilon)
  :=
  \int_{C_{\pm}(\varepsilon)}\tan(N\pi w)g_N(w)\,dw,
\end{equation*}
we have
\begin{equation}\label{eq:G}
  J_N\bigl(E;\exp(\xi/N)\bigr)
  =
  \frac{S_{\gamma}(-\pi-\sqrt{-1}u+\gamma)}{S_{\gamma}(\pi-\sqrt{-1}u-\gamma)}
  \frac{\sqrt{-1}\exp(u/2)N}{2}
  \bigl(G_{+}(N,\varepsilon)+G_{-}(N,\varepsilon)\bigr).
\end{equation}
%%%%%%%%%%%%%%%%%%%%%%%%%%%%%%%%%%%%%%%%%%%%%%%%%%%%%%%%%%%%%%%%%%%%%
\section{Approximating the integral formula}
\label{sec:approximation}
In this section we approximate the integral formula for the colored Jones polynomial obtained in the previous section.
\par
Since $\tan(N\pi w)$ is close to $\sqrt{-1}$ ($-\sqrt{-1}$, respectively) when $\Im(w)$ is positive and large (negative and $|\Im(w)|$ is large, respectively), we can approximate $G_{\pm}(N,\varepsilon)$ by the integral of $g_N(w)$ on $C_{\pm}(\varepsilon)$.
In fact if we write
\begin{equation*}
  G_{\pm}(N,\varepsilon)
  =
  \pm\sqrt{-1}\int_{C_{\pm}(\varepsilon)}g_N(w)\,dw
  +
  \int_{C_{\pm}(\varepsilon)}(\tan(N\pi w)\mp\sqrt{-1})g_N(w)\,dw,
\end{equation*}
then we have the following lemma.
\begin{prop}[see Equation (4.7) in \cite{Andersen/Hansen:JKNOT2006}]\label{prop:4.7}
There exists a positive constant $K_{1,\pm}$ independent of $N$ and $\varepsilon$ such that the following inequality holds:
\begin{equation*}
  \left|
    \int_{C_{\pm}(\varepsilon)}(\tan(N\pi w)\mp\sqrt{-1})g_N(w)\,dw
  \right|
  <
  \frac{K_{1,\pm}}{N}.
\end{equation*}
\end{prop}
A proof is given in Section~\ref{sec:prop:4.7}.
\par
Now we approximate the integral of $g_N(w)$ along $C_{\pm}(\varepsilon)$.
Define
\begin{equation*}
  \Phi(w)
  :=
  \frac{1}{\xi}
  \bigl(
    \Li_2\left(e^{u-\xi w}\right)-\Li_2\left(e^{u+\xi w}\right)
  \bigr)
  -uw.
\end{equation*}
Since $\Li_2$ is analytic in the region $\C\setminus(1,\infty)$, the function $\Phi$ is analytic in the region $\{w\in\C\mid-\frac{2\pi}{u}\Re(w)<\Im(w)<-\frac{2\pi}{u}(\Re(w)-1)\}$ (Figure~\ref{fig:contour}).
\par
\begin{comment}
\begin{rem}
Note that $\Li_2$ is analytic on $\C\setminus\{(1,\infty)\}$ and so $\Phi$ is analytic on $\C\setminus\left\{\dfrac{t+2\pi\sqrt{-1}k}{2\pi\sqrt{-1}+u}\Biggm| k\in\Z,t\in\R\right\}$.
Note also that the set $\left\{\dfrac{t+2\pi\sqrt{-1}k}{2\pi\sqrt{-1}+u}\Biggm| k\in\Z,t\in\R\right\}$ consists of straight lines that cross the points $k$ (on the real axis) and $2k\pi\sqrt{-1}/u$ (on the imaginary axis) for $k\in\Z$.
\end{rem}
\end{comment}
\begin{prop}[see Equation (4.9) in \cite{Andersen/Hansen:JKNOT2006}]
\label{prop:4.9}
Let $p(\varepsilon)$ be any contour in the parallelogram bounded by $C(\varepsilon)$ connecting $\varepsilon$ and $1-\varepsilon$, then there exists a positive constant $K_2$ independent of $N$ and $\varepsilon$ such that the following inequality holds.
\begin{equation*}
  \left|
    \int_{p(\varepsilon)}
    g_N(w)\,dw
    -
    \int_{p(\varepsilon)}
    \exp(N\Phi(w))\,dw
  \right|
  \le
  \frac{K_2\log(N)}{N}
  \max_{w\in p(\varepsilon)}\left\{\exp\bigl(N\Re\Phi(w)\bigr)\right\}.
\end{equation*}
\end{prop}
A proof is given in Section~\ref{sec:prop:4.9}.
\par
We will study the asymptotic behavior of $\int_{C_{\pm}(\varepsilon)}\exp\bigl(N\Phi(w)\bigr)\,dw$ for large $N$.
\par
Since $\Phi(w)$ is analytic in the region $\{w\in\C\mid-\frac{2\pi}{u}\Re(w)<\Im(w)<-\frac{2\pi}{u}(\Re(w)-1)\}$, we have
\begin{equation*}
  \int_{C_{+}(\varepsilon)}\exp\bigl(N\Phi(w)\bigr)\,dw
  =
  \int_{C_{-}(\varepsilon)}\exp\bigl(N\Phi(w)\bigr)\,dw
\end{equation*}
by Cauchy's integral theorem.
\par
We will apply the saddle point method (see for example \cite[\S~7.2]{Marsden/Hoffman:Complex_Analysis}) to approximate the integral $\int_{C_{-}(\varepsilon)}\exp(N\Phi(w))\,dw$.
\par
First we find a solution to the equation $d\,\Phi(w)/d\,w=0$.
Since we have
\begin{equation*}
  \frac{d\,\Phi(w)}{d\,w}
  =
  \log(1-e^{u-\xi w})+\log(1-e^{u+\xi w})-u,
\end{equation*}
a solution to the equation
\begin{equation*}
  e^{\xi w}+e^{-\xi w}
  =
  e^u+e^{-u}-1
\end{equation*}
can be a saddle point that we need.
Put
\begin{align}\label{eq:varphi}
  \varphi(u)
  &:=
  \arccosh(\cosh(u)-1/2)\notag
  \\
  &=
  \log
  \left(
    \frac{1}{2}
    \left(
      e^u+e^{-u}-1-\sqrt{(e^u+e^{-u}+1)(e^u+e^{-u}-3)}
    \right)
  \right),
  \\
  \tilde{\varphi}(u)
  &:=
  \varphi(u)+2\pi\sqrt{-1},\notag
  \\
  \intertext{and}
  w_0
  &:=
  \frac{\tilde{\varphi}(u)}{\xi},\notag
\end{align}
where we choose the square root of $(e^u+e^{-u}+1)(e^u+e^{-u}-3)$ as a positive multiple of $\sqrt{-1}$ and the branch of $\log$ so that $-\pi/3<\Im\varphi(u)<0$.
\begin{rem}\label{rem:varphi}
Note that $\varphi(u)$ and $\tilde{\varphi}(u)$ are purely imaginary since $\left|e^u+e^{-u}-1-\sqrt{(e^u+e^{-u}+1)(e^u+e^{-u}-3)}\right|=4$.
\end{rem}
It is easy to see that $d\,\Phi(w_0)/d\,w=0$.
Since we have
\begin{equation*}
  \Im(w_0)+\frac{2\pi}{u}\Re(w_0)
  =
  \frac{\Im\tilde{\varphi}(u)}{u},
\end{equation*}
$w_0$ is in the domain of $\Phi$.
\par
%%%%%%%%%%%%%%%%%%%%%%%%%%%%%%%%%%%%%%%%%%%%%%%%%%%%%%%%%%%%%%%%%%%%%
\begin{comment}
\begin{figure}[h]
  \includegraphics{Mathematica/11_26_contour.tex_gr4.eps}
  \caption{A path keeping $\Im\Phi(w)$ constant that passes through the saddle point is indicated by a red curve when $u=0$.}
\end{figure}
\begin{figure}[h]
  \includegraphics{Mathematica/11_26_contour.tex_gr8.eps}
  \caption{A path keeping $\Im\Phi(w)$ constant that passes through the saddle point is indicated by a red curve when $u=0.1$.}
  \label{fig:contour01}
\end{figure}
%%%%%%%%%%%%%%%%%%%%%%%%%%%%%%%%%%%%%%%%%%%%%%%%%%%%%%%%%%%%%%%%%%%%%
\begin{figure}[h]
  \includegraphics{Mathematica/11_26_contour.tex_gr12.eps}
  \caption{The path $P$ passing through the saddle point $w_0$, indicated by a black disk, is indicated by the thick curve for $u=0.5$.}
  \label{fig:contour05}
\end{figure}
\end{comment}
%%%%%%%%%%%%%%%%%%%%%%%%%%%%%%%%%%%%%%%%%%%%%%%%%%%%%%%%%%%%%%%%%%%%%
\begin{figure}[h]
  \includegraphics{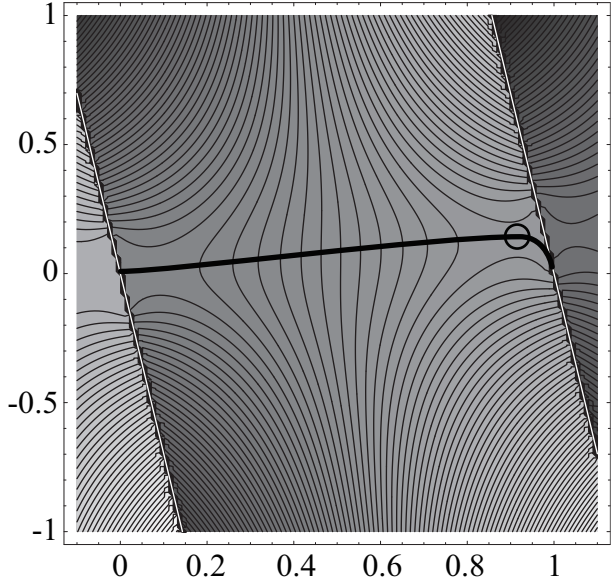}
  \caption{A contour plot of $\Re\Phi(w)$ on the complex plane for $u=0.9$.
  A brighter part is higher than a darker part.
  The path $P$ is indicated by a thick curve and the saddle point $w_0$ is marked by a circle.}
  \label{fig:contour09}
\end{figure}
%%%%%%%%%%%%%%%%%%%%%%%%%%%%%%%%%%%%%%%%%%%%%%%%%%%%%%%%%%%%%%%%%%%%%
\begin{comment}
\begin{figure}[h]
  \includegraphics{Mathematica/11_26_contour.tex_gr20.eps}
  \caption{A path keeping $\Im\Phi(w)$ constant that passes through the saddle point is indicated by a red curve when $u=\log((3+\sqrt{5})/2)$.}
\end{figure}
\begin{figure}[h]
  \includegraphics{Mathematica/11_26_contour.tex_gr24.eps}
  \caption{A path keeping $\Im\Phi(w)$ constant that passes through the saddle point is indicated by a red curve when $u=1$.}
\end{figure}
\begin{figure}[h]
  \includegraphics{Mathematica/11_26_contour.tex_gr28.eps}
  \caption{A path keeping $\Im\Phi(w)$ constant that passes through the saddle point is indicated by a red curve when $u=2$.}
\end{figure}
\end{comment}
%%%%%%%%%%%%%%%%%%%%%%%%%%%%%%%%%%%%%%%%%%%%%%%%%%%%%%%%%%%%%%%%%%%%%
We choose a path $P$ from $\varepsilon$ to $1-\varepsilon$ that passes through $w_0$ so that near $w_0$ it keeps $\Im\Phi(w)$ constant and that $\Re\Phi(w)$ takes its maximum (over all $w$ on $P$) at $w_0$ as indicated in the thick curve in Figures~\ref{fig:contour09}.
Then the integral $\int_{C_-(\varepsilon)}\exp(N\Phi(w))\,dw$ is approximated by the integral near $w_0$ along the path we choose.
More precisely we have
\begin{equation}\label{eq:Phi_saddle}
  \int_{C_-(\varepsilon)}\exp(N\Phi(w))\,dw
  \underset{N\to\infty}{\sim}
  \frac{\sqrt{2\pi}\exp(N\Phi(w_0))}{\sqrt{N}\sqrt{-d^2\,\Phi(w_0)/d\,w^2}}
\end{equation}
from \cite[Theorem~7.2.8]{Marsden/Hoffman:Complex_Analysis}, where the sign of the square root of $-d^2\,\Phi(w_0)/d\,w^2$ is chosen so that
\begin{equation}\label{eq:argument}
  \sqrt{-d^2\,\Phi(w_0)/d\,w^2}\times\text{(tangent of $P$ at $w_0$)}>0.
\end{equation}
Note that
\begin{equation}\label{eq:Phi_0}
  \Phi(w_0)
  =
  \frac{1}{\xi}
  \left(
    \Li_2\left(e^{u-\varphi(u)}\right)
    -
    \Li_2\left(e^{u+\varphi(u)}\right)
    -u\tilde{\varphi}(u)
  \right).
\end{equation}
\par
From Propositions~\ref{prop:4.7} and \ref{prop:4.9}, choosing $P$ as a contour in Proposition~\ref{prop:4.9} we have
\begin{equation*}
\begin{split}
  &\left|
    G_{\pm}(N,\varepsilon)
    \mp
    \sqrt{-1}
    \int_{C_{\pm}(\varepsilon)}\exp\bigl(N\Phi(w)\bigr)\,dw
  \right|
  \\
  \le&
  \left|
    \int_{C_{\pm}(\varepsilon)}(\tan(N\pi w)\mp\sqrt{-1})g_N(w)\,dw
    \pm\sqrt{-1}
    \int_{C_{\pm}(\varepsilon)}\left(g_N(w)-\exp\bigl(N\Phi(w)\bigr)\right)\,dw
  \right|
  \\
  \le&
  \frac{K_{1,\pm}}{N}
  +
  \frac{K_2\log(N)}{N}
  \exp\bigl(N\Re\Phi(w_0)\bigr).
\end{split}
\end{equation*}
From \eqref{eq:Phi_saddle} we have
\begin{equation*}
\begin{split}
  &\lim_{N\to\infty}
  \left|
    \frac{G_{\pm}(N,\varepsilon)}
         {\pm\sqrt{-1}
          \int_{C_{\pm}(\varepsilon)}\exp\bigl(N\Phi(w)\bigr)\,dw}
  -1
  \right|
  \\
  \le&
  \frac{K_{1,\pm}}
       {N
        \left|
          \int_{C_{\pm}(\varepsilon)}\exp\bigl(N\Phi(w)\bigr)\,dw
        \right|}
  +
  \frac{K_2\log(N)}{N}\times
  \frac{\exp\bigl(N\Re\Phi(w_0)\bigr)}
       {\left|
          \int_{C_{\pm}(\varepsilon)}\exp\bigl(N\Phi(w)\bigr)\,dw
        \right|}
  \\
  \underset{N\to\infty}{\longrightarrow}&0.
\end{split}
\end{equation*}
Here we use the following lemma which will be proved in Section~\ref{lem:Phi_0}.
\begin{lem}\label{lem:Phi_0}
The real part of $\Phi(w_0)$ is positive for $0<u<\log\bigl((3+\sqrt{5})/2\bigr)$.
Therefore from \eqref{eq:Phi_saddle} we see that $\int_{C_{\pm}(\varepsilon)}\exp\bigl(N\Phi(w)\bigr)\,dw$ grows exponentially.
\end{lem}
So we have
\begin{equation*}
  G_{\pm}(N,\varepsilon)
  \underset{N\to\infty}{\sim}
  \pm\sqrt{-1}
  \int_{C_{\pm}(\varepsilon)}\exp\bigl(N\Phi(w)\bigr)\,dw.
\end{equation*}
\par
Therefore from \eqref{eq:G} we have
\begin{equation*}
\begin{split}
  &J_N\bigl(K;\exp(\xi/N)\bigr)
  \\
  \underset{N\to\infty}{\sim}&
  \frac{e^{2\pi\sqrt{-1}u N/\xi}}{e^u-1}
%  \times
  \frac{N e^{u/2}}{2}
  \left(
    \int_{C_{-}(\varepsilon)}\exp(N\Phi(w))\,dw
    -
    \int_{C_{+}(\varepsilon)}\exp(N\Phi(w))\,dw
  \right)
  \\
  =&
  \frac{N e^{2\pi\sqrt{-1}u N/\xi}}{2\sinh(u/2)}
  \left(
    \int_{C_{-}(\varepsilon)}\exp(N\Phi(w))\,dw
  \right)
\end{split}
\end{equation*}
from the lemma below.
\begin{lem}\label{lem:S_gamma}
For $\gamma=(2\pi-\sqrt{-1}u)/(2N)$ with positive $u$, we have
\begin{equation*}
  \frac{S_{\gamma}(-\pi-\sqrt{-1}u+\gamma)}{S_{\gamma}(\pi-\sqrt{-1}u-\gamma)}
  =
  \frac{e^{\pi u/\gamma}-1}{e^u-1}.
\end{equation*}
Therefore we have the following asymptotic equivalence:
\begin{equation*}
  \frac{S_{\gamma}(-\pi-\sqrt{-1}u+\gamma)}{S_{\gamma}(\pi-\sqrt{-1}u-\gamma)}
  \underset{N\to\infty}{\sim}
  \frac{e^{2\pi\sqrt{-1}u N/\xi}}{e^u-1}.
\end{equation*}
\end{lem}
A proof of the lemma is given in Section~\ref{sec:S_gamma}.
\begin{rem}\label{rem:S}
When $u=0$, we have
\begin{equation*}
  \frac{S_{\gamma}(-\pi+\gamma)}{S_{\gamma}(\pi-\gamma)}
  =
  N
\end{equation*}
from \cite[P.~492]{Andersen/Hansen:JKNOT2006}.
\end{rem}
\par
Since
\begin{equation*}
  \frac{d^2\,\Phi(w_0)}{d\,w^2}
%  =
%  \frac{\xi e^{u}(e^{-\xi w_0}-e^{\xi w_0})}
%       {(e^{\xi w_0}-e^u)(e^{-\xi w_0}-e^u)}
%  \\
%  &=
%  \frac{\xi e^{u}(e^{-\varphi(u)}-e^{\varphi(u)})}
%       {(e^{\varphi(u)}-e^u)(e^{-\varphi(u)}-e^u)}
  =
  \xi\sqrt{(e^u+e^{-u}+1)(e^u+e^{-u}-3)},
\end{equation*}
we have the following asymptotic equivalence from \eqref{eq:Phi_saddle}:
\begin{equation*}
\begin{split}
  \int_{C_{-}(\varepsilon)}\exp(N\Phi(w))\,dw
  \underset{N\to\infty}{\sim}
  \frac{\sqrt{2\pi}\exp(N\Phi(w_0))}
  {\sqrt{N}\sqrt{-\xi\sqrt{(e^u+e^{-u}+1)(e^u+e^{-u}-3)}}},
\end{split}
\end{equation*}
where we choose the square root of $\sqrt{-\xi\sqrt{(e^u+e^{-u}+1)(e^u+e^{-u}-3)}}$ so that it is in the fourth quadrant from \eqref{eq:argument}.
\par
Therefore we finally have
\begin{equation*}
\begin{split}
  &J_N\bigl(E;\exp(\xi/N)\bigr)
  \\
  \underset{N\to\infty}{\sim}&
  \frac{N e^{2\pi\sqrt{-1}u N/\xi}}{2\sinh(u/2)}
  \frac{\sqrt{2\pi}}{\sqrt{N}\sqrt{-\xi\sqrt{(e^u+e^{-u}+1)(e^u+e^{-u}-3)}}}
  \\
  &
  \times
  \exp
  \left(
    \frac{N}{\xi}
    \left(
      \Li_2\left(e^{u-\varphi(u)}\right)
      -
      \Li_2\left(e^{u+\varphi(u)}\right)
      -
      u\tilde{\varphi}(u)
    \right)
  \right)
  \\
  =&
  \frac{\sqrt{\pi}}{2\sinh(u/2)}
  \sqrt{
    \frac{-2}{\sqrt{(e^u+e^{-u}+1)(e^u+e^{-u}-3)}}
       }
  \sqrt{\frac{N}{\xi}}
  \exp\left(\frac{N}{\xi}S(u)\right).
\end{split}
\end{equation*}
Here we put
\begin{equation*}
  S(u)
  :=
  \Li_2\left(e^{u-\varphi(u)}\right)
  -
  \Li_2\left(e^{u+\varphi(u)}\right)
  -
  u\varphi(u).
\end{equation*}
\begin{rem}
When $u=0$, we have
\begin{equation*}
\begin{split}
  \int_{C_{-}(\varepsilon)}\exp(N\Phi(w))\,dw
  \underset{N\to\infty}{\sim}
  \frac{\exp(N\Phi(w_0))}
  {\sqrt{N}3^{1/4}}.
\end{split}
\end{equation*}
Since $w_0=5/6$ in this case we have
\begin{equation*}
\begin{split}
  \Phi(w_0)
  &=
  \frac{1}{2\pi\sqrt{-1}}
  \left(
    \Li_2(e^{-5\pi\sqrt{-1}/3})-\Li_2(e^{5\pi\sqrt{-1}/3})
  \right)
  \\
  &=
  \frac{6\sqrt{-1}\Lambda(\pi/3)}{2\pi\sqrt{-1}}
  =
  \frac{\Vol(E)}{2\pi}.
\end{split}
\end{equation*}
Therefore from Remark~\ref{rem:S}, we have Theorem~\ref{thm:Andersen/Hansen}.
\end{rem}
%%%%%%%%%%%%%%%%%%%%%%%%%%%%%%%%%%%%%%%%%%%%%%%%%%%%%%%%%%%%%%%%%%%%%
\begin{comment}
\begin{rem}
When $u=\log\bigl((3+\sqrt{5})/2\bigr)$, $d^2\,\Phi(u)/d\,u^2$ also vanishes.
Since $\varphi(u)=0$ in this case, we have $w_0=2\pi\sqrt{-1}/\xi$ and
\begin{equation*}
\begin{split}
  \Phi(w)
  &=
  \Phi(w_0)
  +
  \frac{1}{3!}
  \frac{d^3\,\Phi(w_0)}{d\,w^3}(w-w_0)^3
  +\dots
  \\
  &=
  \frac{-\xi^2e^{u+\xi w_0}(1+e^{2u}+e^{2\xi w_0}-4e^{u+\xi w_0}+e^{2(u+\xi w_0)})}
       {6(e^{u}-e^{\xi w_0})^2(e^{u+\xi w_0}-1)^2}
  (w-w_0)^3
  \\
  &\quad+\text{higher order terms of $(w-w_0)$}
  \\
  &=
  \frac{-\xi^2e^{u}}
       {3(e^{u}-1)^2}
  (w-w_0)^3
  +\text{higher order terms of $(w-w_0)$}
  \\
  &=
  -\frac{\xi^2}{3}(w-w_0)^3
  +\text{higher order terms of $(w-w_0)$}.
\end{split}
\end{equation*}
Therefore putting $z:=(N\xi^2/3)^{1/3}(w-w_0)$ we have
\begin{equation*}
  \int_{C_{-}(\varepsilon)}\exp(N\Phi(w))\,dw
  =
  \left(\frac{3}{N\xi^2}\right)^{1/3}
  \int_{C'}\exp[-z^3+\text{higher order terms of $z$}]\,dz
\end{equation*}
\end{rem}
\end{comment}
%%%%%%%%%%%%%%%%%%%%%%%%%%%%%%%%%%%%%%%%%%%%%%%%%%%%%%%%%%%%%%%%%%%%%%
\section{Topological interpretations of $S(u)$ and $T(u)$}
\label{sec:interpretation}
In this section we describe topological interpretations of $S(u)$ and $T(u)$.
\subsection{Representation}
Let $x$ and $y$ be the Wirtinger generators of the fundamental group $\pi_1(S^3\setminus{E})$ (with a base point above the paper) of the complement of the figure-eight knot $E$ depicted in Figure~\ref{fig:fig8_pi1}.
\begin{figure}[h]
  \includegraphics[scale=0.3]{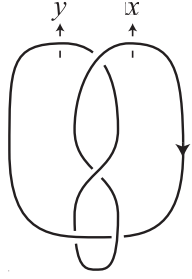}
  \caption{Generators of the fundamental group of the complement of the figure-eight knot}
  \label{fig:fig8_pi1}
\end{figure}
The group $\pi_1(S^3\setminus{E})$ has the following presentation:
\begin{equation*}
  \pi_1(S^3\setminus{E})
  =
  \langle
    x,y
    \mid
    xy^{-1}x^{-1}yx=yxy^{-1}x^{-1}y
  \rangle.
\end{equation*}
Due to \cite{Riley:QUAJM31984} any non-abelian representation $\rho$ of $\pi_1(S^3\setminus{E})$ into $SL(2;\C)$ is, up to conjugation, given as follows:
\begin{align*}
  \rho(x)&
  :=
  \begin{pmatrix}
    m^{1/2} & 1 \\
    0       & m^{-1/2}
  \end{pmatrix},
  \\
  \rho(y)
  &:=
  \begin{pmatrix}
    m^{1/2} & 0 \\
    -d      & m^{-1/2}
  \end{pmatrix},
\end{align*}
where
\begin{equation*}
  d=
  \frac{1}{2}
  \left(
    m+m^{-1}-3\
    \pm
    \sqrt{(m+m^{-1}+1)(m+m^{-1}-3)}
  \right).
\end{equation*}
Since the longitude $\lambda$ is given by $xy^{-1}xyx^{-2}yxy^{-1}x^{-1}$ if we read it off from the top right, we have
\begin{equation*}
  \rho(\lambda)
  =
  \begin{pmatrix}
    \ell(m)^{\pm1}& (m^{1/2}+m^{-1/2})\sqrt{(m+m^{-1}+1)(m+m^{-1}-3)} \\
    0             & \ell(m)^{\mp1}
  \end{pmatrix},
\end{equation*}
where
\begin{equation*}
  \ell(m)
  :=
  \frac{m^2-m-2-m^{-1}+m^{-2}}{2}
  +
  \frac{(m-m^{-1})}{2}\sqrt{(m+m^{-1}+1)(m+m^{-1}-3)}.
\end{equation*}
See also \cite[Section~3.1]{Murakami:ACTMV2008}.
\par
Let $\rho_u$ be the representation given by putting $m:=e^u$.
We introduce a parameter $v$ so that $\ell(e^u)=e^{v/2}$.
Since we assume that $0<u<\log\bigl((3+\sqrt{5})/2\bigr)$, we have $2<e^u+e^{-u}<3$.
Therefore we have
\begin{equation*}
  |\ell(e^u)|^2
  =
  \frac{1}{4}
  (e^{2u}-e^u-2-e^{-u}+e^{-2u})^2
  -
  \frac{1}{4}
  (e^u-e^{-u})^2(e^u+e^{-u}+1)(e^u+e^{-u}-3)
  =
  1
\end{equation*}
and so $v$ is purely imaginary.
\par
The representation $\rho_u$ gives an incomplete hyperbolic structure to $S^3\setminus{E}$ and its completion is the generalized Dehn surgery \cite{Thurston:GT3M} with parameter $(p,q)$ with $pu+qv=2\pi\sqrt{-1}$.
Since $v$ is purely imaginary $p=0$ and $q=2\pi\sqrt{-1}/v$.
Therefore the completion is a cone manifold whose underlying space is the $0$-surgery of $S^3$ along the figure-eight knot with singularity the core of the surgery and with cone angle $\alpha=\Im(v)=v/\sqrt{-1}$.
Note that when $u=0$, the cone angle is $2\pi$ and when $u=\log\bigl((3+\sqrt{5})/2\bigr)$, the cone angle is $0$.
See \cite{Hilden/Lozano/Montesinos:JMASU1995} for more details about the geometric structure of this manifold.
\par
In the following two subsections we will calculate the Reidemeister torsion and the Chern--Simons invariant associated with $\rho_u$.
\subsection{Reidemeister torsion}
From \cite[P.~113]{Porti:MAMCAU1997} (see also \cite[\S~6.3]{Dubois:CANMB2006}) the cohomological Reidemeister torsion $T^{E}_{\lambda}(\rho_u)$ associated with the {\em longitude} $\lambda$ twisted by the adjoint action of the representation $\rho_u$ is given by
\begin{equation*}
  T^{E}_{\lambda}(\rho_u)
  =
  \frac{1}{\sqrt{17+4\Tr\bigl(\rho_u(\lambda)\bigr)}}
  =
  \frac{1}{2m+2m^{-1}-1}
\end{equation*}
up to sign, where $\Tr$ means the trace.
Note that since in \cite{Porti:MAMCAU1997} Porti uses homological Reidemeister torsion, we need to take the inverse.
\par
From \cite[Th{\'e}or{\`e}me~4.1]{Porti:MAMCAU1997} the Reidemeister torsion $T^{E}_{\mu}(\rho_u)$ associated with the {\em meridian} $\mu$ is given by 
\begin{equation*}
  T^{E}_{\mu}(\rho_u)
  =
  \pm\frac{\partial\, v}{\partial\, u}T^{E}_{\lambda}(\rho_u).
\end{equation*}
Since $\ell(e^u)=e^{v/2}$, we have
\begin{equation*}
  \pm
  T^{E}_{\mu}(\rho_u)
  =
  \frac{\partial\,\bigl(2\log\ell(e^{u})\bigr)}{\partial\, u}
  \times
  \frac{1}{2e^u+2e^{-u}-1}
  =
  \frac{2}{\sqrt{(e^u+e^{-u}+1)(e^{u}+e^{-u}-3)}}.
\end{equation*}
Therefore $T(u)$ that appears in Theorem~\ref{thm:main} coincides with $T^{E}_{\mu}(\rho_u)$ up to sign.
\subsection{Chern--Simons invariant}
Let $M$ be a closed three-manifold and $\rho\colon\pi_1(M)\to\SL(2;\C)$ a representation.
Then the Chern--Simons invariant $\cs_{M}(\rho)$ is defined as
\begin{equation*}
  \frac{1}{8\pi^2}\int_{M}\Tr(A\wedge dA+\frac{2}{3}A\wedge A\wedge A)
  \in\C/\Z,
\end{equation*}
where $A$ is the $\mathfrak{sl}(2;\C)$-valued $1$-form on $M$ with $dA+A\wedge A=0$ such that $\rho$ is given as the holonomy representation induced by the flat connection on $M\times\SL(2;\C)$ defined by $A$ .
\par
In \cite{Kirk/Klassen:COMMP93} Kirk and Klassen defined the $\SL(2;\C)$ Chern--Simons invariant $\cs_M(\rho)$ for a three-manifold with boundary.
It is a triple $[\alpha,\beta;z]$ of complex numbers modulo the following relation:
\begin{equation*}
  [\alpha,\beta;z]
  =
  [\alpha+1,\beta;z\exp(-8\pi\sqrt{-1}\beta)]
  =
  [\alpha,\beta+1;z\exp(8\pi\sqrt{-1}\alpha)]
  =
  [-\alpha,-\beta;z].
\end{equation*}
For $i=1,2$, let $M_i$ be a three-manifold with boundary $\partial{M_i}$ a torus and put $M:=M_1\bigcup_{\partial{M_1}=-\partial{M_2}}M_2$.
For a representation $\rho\colon\pi_1(M)\to SL(2;\C)$, put $\rho_i:=\rho\bigr|_{M_i}$.
If $\cs_{M_i}(\rho_i)=[\alpha,\beta;z_i]$ then  $\cs_M(\rho)$ is given by $z_1z_2$.
\par
Define the Chern--Simons invariant $\CS_u(K)$ for a knot $K$ to be
\begin{equation*}
  \cs_{S^3\setminus{E}}(\rho_u)
  =
  \left[
    \frac{u}{4\pi\sqrt{-1}},
    \frac{v}{4\pi\sqrt{-1}};
    \exp
    \left(
      \frac{2}{\pi\sqrt{-1}}\CS_{u}(K)
    \right)
  \right].
\end{equation*}
Then as described in \cite{Hikami/Murakami:Bonn} we have
\begin{equation*}
   \CS_{u}(E)
   =
   S(u)-\pi\sqrt{-1}u-\frac{uv}{4}.
\end{equation*}
Note that we are using the $\operatorname{PSL}(2;\C)$ normalization of the Chern--Simons invariant \cite[P.~543]{Kirk/Klassen:COMMP93}.
So the function $f(u)$ in \cite{Neumann/Zagier:TOPOL85} is $-\CS_u(E)/4$ (up to a constant) and the function $f(u)$ in \cite[P.~543]{Kirk/Klassen:COMMP93} and \cite{Yoshida:INVEM85} is $-2\sqrt{-1}\CS_u(E)/\pi$.
%%%%%%%%%%%%%%%%%%%%%%%%%%%%%%%%%%%%%%%%%%%%%%%%%%%%%%%%%%%%%%%%%%%%%%
%\appendix
\section{Calculation of $S_{\gamma}$}
\label{sec:S_gamma}
In this section we first show the analyticity of $S_{\gamma}$ and then calculate its special values.
\begin{lem}\label{lem:analyticity}
If a complex number $\gamma$ has the positive real part, then $S_{\gamma}(z)$ is an analytic function in $\{z\in\C\mid|\Re(z)|<\pi+\Re(\gamma)\}$.
\end{lem}
\begin{proof}
Put
\begin{equation*}
  L_{\gamma}(t)
  :=
    \frac{e^{zt}}{t\sinh(\pi t)\sinh(\gamma t)}.
\end{equation*}
We will show that the improper integrals $\int_{R}^{\infty}L_{\gamma}t\,dt$ and $\int_{-\infty}^{R}L_{\gamma}t\,dt$ converge.
\par
Putting $z:=x+\sqrt{-1}y$ and $\gamma:=a+\sqrt{-1}b$ for real numbers $x,y,a,b$ with $a>0$ we have
\begin{equation*}
  \Re(L_{\gamma}(t))
  =
  \frac
  {-2e^{xt}\bigl(\cosh(at)\sin(bt)\sin(ty)+\sinh(at)\cos(bt)\cos(yt)\bigr)}
  {t\sinh(\pi t)\bigl(\cos(2bt)-\cosh(2at)\bigr)}
\end{equation*}
and
\begin{equation*}
  \Im(L_{\gamma}(t))
  =
  \frac
  {2e^{xt}\bigl(\cosh(at)\sin(bt)\cos(ty)-\sinh(at)\cos(bt)\sin(yt)\bigr)}
  {t\sinh(\pi t)\bigl(\cos(2bt)-\cosh(2at)\bigr)}
\end{equation*}
for $t\in\R$.
\par
If $t$ is positive and sufficiently large we have
\begin{equation*}
\begin{split}
  |\Re(L_{\gamma}(t))|
  &\le
  \frac{2e^{xt}e^{at}}{t\sinh(\pi t)(\cosh(2at)-1)}
  \\
  &=
  \frac{8e^{(x-\pi-a)t}}{t(1-e^{-2\pi t})(1+e^{-4at}-2e^{-2at})}
\end{split}
\end{equation*}
and similarly we have
\begin{equation*}
  |\Im(L_{\gamma}(t))|
  \le
  \frac{8e^{(x-\pi-a)t}}{t(1-e^{-2\pi t})(1+e^{-4at}-2e^{-2at})}.
\end{equation*}
Therefore the integral $\int_{R}^{\infty}L_{\gamma}t\,dt$ converges since $x<\pi+a$.
\par
If $t$ is negative and $|t|$ is sufficiently large we have
\begin{equation*}
\begin{split}
  |\Re(L_{\gamma}(t))|
  &\le
  \frac{2e^{(x-a)t}}{t\sinh(\pi t)(\cosh(2at)-1)}
  \\
  &=
  \frac{8e^{(x+\pi+a)t}}{t(e^{2\pi t}-1)(e^{4at}+1-2e^{2at})}
\end{split}
\end{equation*}
and similarly we have
\begin{equation*}
  |\Im(L_{\gamma}(t))|
  \le
  \frac{8e^{(x+\pi+a)t}}{t(e^{2\pi t}-1)(e^{4at}+1-2e^{2at})}.
\end{equation*}
Therefore the integral $\int_{-\infty}^{R}L_{\gamma}t\,dt$ converges since $x>-\pi-a$.
\end{proof}
Next we prove Lemma~\ref{lem:S_gamma}.
\begin{proof}[Proof of Lemma~\ref{lem:S_gamma}]
By the definition we have
\begin{equation*}
\begin{split}
  &\frac{S_{\gamma}(-\pi-\sqrt{-1}u+\gamma)}{S_{\gamma}(\pi-\sqrt{-1}u-\gamma)}
  \\
  =&
  \exp
  \left(
    \frac{1}{4}
    \int_{C_R}
    \left(
      \frac{e^{(-\pi-\sqrt{-1}u+\gamma)t}}{\sinh(\pi t)\sinh(\gamma t)}
      -
      \frac{e^{(\pi-\sqrt{-1}u-\gamma)t}}{\sinh(\pi t)\sinh(\gamma t)}
    \right)
    \frac{dt}{t}
  \right)
  \\
  =&
  \exp
  \left(
    \frac{1}{2}
    \int_{C_R}
    e^{-\sqrt{-1}ut}
    \frac{\sinh(-\pi t+\gamma t)}{\sinh(\pi t)\sinh(\gamma t)}
    \frac{dt}{t}
  \right)
  \\
  =&
  \exp
  \left(
    \frac{1}{2}
    \int_{C_R}
    \frac{e^{-\sqrt{-1}ut}\coth(\pi t)}{t}
    \,dt
    -
    \frac{1}{2}
    \int_{C_R}
    \frac{e^{-\sqrt{-1}ut}\coth(\gamma t)}{t}
    \,dt
  \right).
\end{split}
\end{equation*}
We will calculate the integral $\int_{C_R}e^{-\sqrt{-1}ut}\coth(\kappa t)/t\,dt$ for a complex number $\kappa$ with $\Re(\kappa)>0$ and $\Im(\kappa)\le0$.
\par
Put $\kappa:=\alpha-\beta\sqrt{-1}$ with $\alpha>0$ and $\beta\ge0$.
For a positive number $r$, let $U_1$ be the segment connecting $r$ and $r+r'\beta/\alpha-r'\sqrt{-1}$, $U_2$ be the segment connecting $r+r'\beta/\alpha-r'\sqrt{-1}$ and $-r+r'\beta/\alpha-r'\sqrt{-1}$, and $U_3$ be the segment connecting $-r+r'\beta/\alpha-r'\sqrt{-1}$ and $-r$ with $r':=(n+1/2)\pi\alpha/|\kappa|^2$, where $n:=\lfloor{r|\kappa|^2/(\pi\alpha)}\rfloor$.
Here $\lfloor{x}\rfloor$ is the largest integer that does not exceed $x$.
Note that $r-\pi\alpha/(2|\kappa|^2)<r'\le r+\pi\alpha/(2|\kappa|^2)$.
We use $r'$ instead of $r$ just to avoid the poles of $\coth(\kappa t)$.
Then we have
\begin{equation*}
\begin{split}
  &\left|
    \int_{U_1}
    \frac{e^{-\sqrt{-1}ut}}{t}
    \coth(\kappa t)
    \,dt
  \right|
  \\
  \le&
  \int_{0}^{r'}
  \left|
    \frac{e^{-\sqrt{-1}u(r+s\beta/\alpha-s\sqrt{-1})}}
         {r+s\beta/\alpha-s\sqrt{-1}}
  \right|
  \left|
    \coth\bigl((\alpha-\beta\sqrt{-1})(r+s\beta/\alpha-s\sqrt{-1})\bigr)
  \right|
  \,ds
  \\
  \le&
  \frac{1}{r}
  \int_{0}^{r'}
  e^{-us}
  \left|
    \coth\bigl(r\alpha-\sqrt{-1}(s\alpha+r\beta+s\beta^2/\alpha)\bigr)
  \right|
  \,ds
  \\
  \le&
  \frac{1}{r}
  \int_{0}^{r'}
  e^{-us}
  \,ds
  \\
  =&
  \frac{1}{ur}
  (1-e^{-ur'})
  \xrightarrow{r\to\infty}0.
\end{split}
\end{equation*}
Similarly we have
\begin{equation*}
\begin{split}
  &\left|
    \int_{U_3}
    \frac{e^{-\sqrt{-1}ut}}{t}
    \coth(\kappa t)
    \,dt
  \right|
  \\
  \le&
  \int_{0}^{r'}
  \left|
    \frac{e^{-\sqrt{-1}u(-r+s\beta/\alpha-s\sqrt{-1})}}
         {-r+s\beta/\alpha-s\sqrt{-1}}
  \right|
  \left|
    \cosh\bigl((\alpha-\beta\sqrt{-1})(-r+s\beta/\alpha-s\sqrt{-1})\bigr)
  \right|
  \,ds
  \\
  \le&
  \frac{|\kappa|}{r\alpha}
  \int_{0}^{r'}
  e^{-us}
  \left|
    \cosh\bigl(-r\alpha-\sqrt{-1}(s\alpha-r\beta+s\beta^2/\alpha)\bigr)
  \right|
  \,ds
  \\
  \le&
  \frac{|\kappa|}{ur\alpha}
  (1-e^{-ur'})
  \xrightarrow{r\to\infty}0.
\end{split}
\end{equation*}
We also have
\begin{equation*}
\begin{split}
  &\left|
    \int_{U_2}
    \frac{e^{-\sqrt{-1}ut}}{t}
    \coth(\kappa t)
    \,dt
  \right|
  \\
  \le&
  \int_{-r}^{r}
  \left|
    \frac{e^{-\sqrt{-1}u(s+r'\beta/\alpha-r'\sqrt{-1})}}
         {s+r'\beta/\alpha-r'\sqrt{-1}}
  \right|
  \left|
    \coth\bigl((\alpha-\beta\sqrt{-1})(s+r'\beta/\alpha-r'\sqrt{-1})\bigr)
  \right|
  \,ds
  \\
  \le&
  \frac{e^{-ur'}}{r'}
  \int_{-r}^{r}
  \left|
   \coth\bigl(s\alpha-\sqrt{-1}(r'\alpha+s\beta+r'\beta^2/\alpha)\bigr)
  \right|
  \,ds
  \\
  =&
  \frac{e^{-ur'}}{r'}
  \int_{-r}^{r}
  \left|
   \coth
   \left(
     s\alpha-\sqrt{-1}\bigl((n+1/2)\pi+s\beta\bigr)
   \right)
  \right|
  \,ds
  \\
  =&
  \frac{e^{-ur'}}{r'}
  \int_{-r}^{r}
  \left|
    \tanh\bigl(\kappa s\bigr)
  \right|
  \,ds
  \\
  &\text{($\delta:=\max_{-1\le s\le1}|\tanh(\kappa s)|>0$)}
  \\
  \le&
  \frac{e^{-ur'}}{r'}
  \left(
    2\delta
    +
    \int_{-r}^{-1}
    \frac{|e^{\kappa s}-e^{-\kappa s}|}{|e^{\kappa s}+e^{-\kappa s}|}
    \,ds
    +
    \int_{1}^{r}
    \frac{|e^{\kappa s}-e^{-\kappa s}|}{|e^{\kappa s}+e^{-\kappa s}|}
    \,ds
  \right)
  \\
  \le&
  \frac{e^{-ur'}}{r'}
  \left(
    2\delta
    +
    \int_{-r}^{-1}
    \frac{|e^{\kappa s}|+|e^{-\kappa s}|}
         {\bigl||e^{\kappa s}|-|e^{-\kappa s}|\bigr|}
    \,ds
    +
    \int_{1}^{r}
    \frac{|e^{\kappa s}|+|e^{-\kappa s}|}
         {\bigl||e^{\kappa s}|-|e^{-\kappa s}|\bigr|}
    \,ds
  \right)
  \\
  =&
  \frac{2e^{-ur'}}{r'}
  \left(
    \delta
    +
    \int_{1}^{r}|\coth(\alpha s)|\,ds
  \right)
  \\
  =&
  \frac{2e^{-ur'}}{r'}
  \left(
    \delta
    +
    \frac{\log(\sinh(\alpha r))-\log(\sinh(\alpha))}{\alpha}
  \right)
  \xrightarrow{r\to\infty}0.
\end{split}
\end{equation*}
Therefore we have
\begin{equation*}
\begin{split}
  \int_{C_R}\frac{e^{-\sqrt{-1}ut}}{t}\coth(\kappa t)\,dt
  &=
  2\pi\sqrt{-1}
  \sum_{l=1}^{\infty}
  \Res
  \left(
    \frac{e^{-\sqrt{-1}ut}}{t}\coth(\kappa t);
    t=\frac{l\pi\sqrt{-1}}{\kappa}
  \right)
  \\
  &=
  2\pi\sqrt{-1}
  \sum_{l=1}^{\infty}
  \frac{e^{lu\pi/\kappa}}{l\pi\sqrt{-1}}
  \\
  &=
  -2\log(1-e^{u\pi/\kappa})
\end{split}
\end{equation*}
and so we have
\begin{equation*}
  \frac{S_{\gamma}(-\pi-\sqrt{-1}u+\gamma)}{S_{\gamma}(\pi-\sqrt{-1}u-\gamma)}
  =
  \frac{1-e^{u\pi/\gamma}}{1-e^u}.
\end{equation*}
\end{proof}
%%%%%%%%%%%%%%%%%%%%%%%%%%%%%%%%%%%%%%%%%%%%%%%%%%%%%%%%%%%%%
\section{Proof of Proposition~\ref{prop:4.7}}
\label{sec:prop:4.7}
In this section we follow \cite[Appendix~A]{Andersen/Hansen:JKNOT2006} to show Proposition~\ref{prop:4.7}.
\par
From \cite[\S~4.1]{Andersen/Hansen:JKNOT2006} we have the following integral expression for $|\Re(z)|<\pi$, or $|\Re(z)|=\pi$ and $\Im(z)\ge0$:
\begin{equation*}
  \frac{1}{2\sqrt{-1}}
  \Li_2(-e^{\sqrt{-1}z})
  =
  \frac{1}{4}
  \int_{C_R}\frac{e^{zt}}{t^2\sinh(\pi t)}\,dt.
\end{equation*}
Therefore we have
\begin{equation*}
\begin{split}
  S_{\gamma}(z)
  &=
  \exp
  \left(
    \frac{1}{2\sqrt{-1}\gamma}
    \Li_2(-e^{\sqrt{-1}z})
    +
    I_{\gamma}(z)
  \right)
  \\
  &=
  \exp
  \left(
    \frac{N}{\xi}
    \Li_2(-e^{\sqrt{-1}z})
    +
    I_{\gamma}(z)
  \right),
\end{split}
\end{equation*}
where
\begin{equation*}
  I_{\gamma}(z)
  :=
  \frac{1}{4}
  \int_{C_R}\frac{e^{zt}}{t\sinh(\pi t)}
  \left(
    \frac{1}{\sinh(\gamma t)}-\frac{1}{\gamma t}
  \right)
  \,dt
\end{equation*}
(see \cite[Equation~(4.2)]{Andersen/Hansen:JKNOT2006}).
Note that $I_{\gamma}(z)$ is defined for $z$ with $|\Re(z)|\le\pi$ (\cite[Appendix~A]{Andersen/Hansen:JKNOT2006}).
\par
Then we have from \eqref{eq:g_N_def}
\begin{equation}\label{eq:g_N}
\begin{split}
  &g_N(z)
  \\
  =&
  \exp
  \left(
    \frac{N}{\xi}
    \left(
      \Li_2(-e^{\sqrt{-1}\pi+u-\xi z})
      -
      \Li_2(-e^{-\sqrt{-1}\pi+u+\xi z})
    \right)
    -Nuz
  \right)
  \\
  &\times
  \exp
  \left(
    I_{\gamma}(\pi-\sqrt{-1}u+\sqrt{-1}\xi z)
    -
    I_{\gamma}(-\pi-\sqrt{-1}u-\sqrt{-1}\xi z)
  \right)
  \\
  =&
  \exp\bigl(N\Phi(z)\bigr)
  \exp
  \bigl(
    I_{\gamma}(\pi-\sqrt{-1}u+\sqrt{-1}\xi z)
    -
    I_{\gamma}(-\pi-\sqrt{-1}u-\sqrt{-1}\xi z)
  \bigr).
\end{split}
\end{equation}
\par
We first give an estimation for $|I_{\gamma}(z)|$.
\begin{lem}[see Lemma~3 in \cite{Andersen/Hansen:JKNOT2006}]\label{lem:Lemma3}
If $|\Re(z)|<\pi$, then we have
\begin{equation*}
  |I_{\gamma}(z)|
  \le
  A
  \left(
    \frac{1}{\pi-\Re(z)}+\frac{1}{\pi+\Re(z)}
  \right)
  |\gamma|
  +
  B
  \left(
    1+e^{-\Im(z)R}
  \right)
  |\gamma|
\end{equation*}
If $|\Re(z)|\le\pi$, then we have
\begin{equation*}
  |I_{\gamma}(z)|
  \le
  2A+B(1+e^{-\Im(z)R})|\gamma|.
\end{equation*}
\end{lem}
\begin{proof}
We follow \cite[Appendix~A]{Andersen/Hansen:JKNOT2006}.
\par
Put
\begin{equation*}
  \psi(w)
  :=
  \frac{1}{\sinh(w)}-\frac{1}{w},
\end{equation*}
which is holomorphic in the open disk $D_0(\pi)$ with center $0$ and radius $\pi$.
Note that
\begin{equation}\label{eq:psi}
  \psi(w)
  =
  \frac{w-\sinh(w)}{w\sinh(w)}
  =
  -w
  \frac{h(w)}{k(w)}
\end{equation}
for entire functions $h(w)$ and $k(w)$ with
\begin{align*}
  h(w)&=\sum_{j=0}^{\infty}\frac{w^{2j}}{(2j+3)!}
  \intertext{and}
  k(w)&=\sum_{j=0}^{\infty}\frac{w^{2j}}{(2j+1)!}
\end{align*}
when $|w|<\pi$.
Therefore there exists $\delta>0$ such that $\min_{|w|\le\delta}|\psi(w)/w|=D>0$ since $\lim_{w\to0}\psi(w)/w=1/6$.
\par
Put $C_{\delta}:=[-\delta/|\gamma|,-R]\cup\Omega_R\cup[R,\delta/|\gamma|]$.
Consider the following integrals $I_0(z)$ and $I_1(z)$ so that $I_{\gamma}(z)=I_0(z)+I_1(z)$:
\begin{align*}
  I_0(z)
  &:=
  \frac{1}{4}
  \int_{C_{\delta}}
  \frac{\exp(zt)}{t\sinh(\pi t)}\psi(\gamma t)
  \,dt,
  \\
  I_1(z)
  &:=
  \frac{1}{4}
  \int_{-\infty}^{-\delta/|\gamma|}
  \frac{\exp(zt)}{t\sinh(\pi t)}\psi(\gamma t)
  \,dt
  +
  \frac{1}{4}
  \int_{\delta/|\gamma|}^{\infty}
  \frac{\exp(zt)}{t\sinh(\pi t)}\psi(\gamma t)
  \,dt.
\end{align*}
Since $\lim_{w\to\infty}w\psi(w)=0$, $\psi(w)$ has poles at $w=m\pi\sqrt{-1}$ ($m=\pm1,\pm2,\dots$), and $\Im(\gamma)\ne0$, we have $|\gamma t\psi(\gamma t)|\le E$ for a positive number $E$.
Note that $E$ depends only on the argument of $\gamma$ and so only on $\xi=\gamma\times 2\sqrt{-1}N$.
So we have
\begin{equation*}
  |\psi(\gamma t)|
  \le
  \frac{E}{|\gamma t|}
\end{equation*}
for any $t$.
Therefore we can apply the argument (replacing $\gamma$ there with $|\gamma|/E$ and $a$ with $\delta/|\gamma|$) in \cite[Page~532]{Andersen/Hansen:JKNOT2006} to conclude
\begin{equation*}
  |I_1(z)|
  \le
  \frac{E}{\delta\left(1-e^{-2\pi|\delta|/|\gamma|}\right)}
\end{equation*}
if $|\Re(z)|\le\pi$ and
\begin{equation*}
  |I_1(z)|
  \le
  \frac{E|\gamma|}
       {2\delta^2\left(1-e^{-2\pi\delta/|\gamma|}\right)}
  \left(
    \frac{e^{-(\pi-\Re(z))\delta/|\gamma|)}}{\pi-\Re(z)}
    +
    \frac{e^{-(\pi+\Re(z))\delta/|\gamma|)}}{\pi+\Re(z)}
  \right)
\end{equation*}
if $|\Re(z)|<\pi$.
\par
\begin{comment}
If $t\in {C'}_{R}$, then $|\gamma t|\le\delta$ and so we have $|\psi(\gamma t)|\le D|\gamma||t|$.
So we have
\begin{equation*}
  |I_0(z)|
  \le
  \frac{|\gamma|}{4}
  \int_{C_{\delta}}
  \left|
    \frac{\exp(zt)}{\sinh(\pi t)}
  \right|\,dt.
\end{equation*}
\par
\end{comment}
Next we estimate $I_0(z)$.
Let $M(z,R)$ be the maximum of $\left|\frac{e^{z t}}{t\sinh(\pi t)}\psi(\gamma t)\right|$ for $t\in\Omega_R$.
Then we have
\begin{equation*}
  \left|
    \int_{\Omega_R}
    \frac{\exp(zt)}{t\sinh(\pi t)}\psi(\gamma t)
    \,dt
  \right|
  \le
  \pi R M(z,R).
\end{equation*}
\par
From \eqref{eq:psi} we have
\begin{equation*}
\begin{split}
  M(z,R)
  &=
  \max_{t\in\Omega_R}
  \left|\frac{e^{z t}}{t\sinh(\pi t)}\psi(\gamma t)\right|
  \\
  &=
  \max_{t\in\Omega_R}
  |\gamma|
  \left|\frac{e^{z t}}{\sinh(\pi t)}\right|
  \left|\frac{h(\gamma t)}{k(\gamma t)}\right|
\end{split}
\end{equation*}
since $|\gamma t|=|\gamma|R<\pi$ when $t\in\Omega_R$ and $R<1$.
Then putting $L(R):=\max_{t\in\Omega_R}|h(z)/k(z)|$ and $N(z,R):=\max_{t\in\Omega_R}\left|e^{zt}/(e^{\pi t}-e^{-\pi t})\right|$ we can apply the argument in \cite[Page~533]{Andersen/Hansen:JKNOT2006} to have the following estimation.
\begin{equation*}
  \left|
    \frac{1}{4}
    \int_{\Omega_R}
    \frac{e^{zt}}{t\sinh(\pi t)}\psi(\gamma t)\,dt
  \right|
  \le
  |\gamma|B(1+e^{-\Im(z)R})
\end{equation*}
for a constant $B$ (depending only on $R$).
\par
Now we will estimate the rest of $I_0(z)$.
We have
\begin{equation*}
\begin{split}
  &\left|
    \int_{-\delta/|\gamma|}^{-R}\frac{e^{zt}}{t\sinh(\pi t)}\psi(\gamma t)\,dt
    +
    \int_{R}^{\delta/|\gamma|}\frac{e^{zt}}{t\sinh(\pi t)}\psi(\gamma t)\,dt
  \right|
  \\
  \le&
  \int_{-\delta/|\gamma}^{-R}
  \frac{e^{\Re(z)t}}{|t||\sinh(\pi t)|}|\psi(\gamma t)|\,dt
  +
  \int_{R}^{\delta/|\gamma}
  \frac{e^{\Re(z)t}}{|t||\sinh(\pi t)|}|\psi(\gamma t)|\,dt
  \\
  =&
  \int_{R}^{\delta/|\gamma}
  \frac{e^{\Re(z)t}+e^{-\Re(z)t}}{t\sinh(\pi t)}
  |\psi(\gamma t)|\,dt
  \\
  =&
  2|\gamma|
  \int_{R}^{\delta/|\gamma}
  \frac{e^{\Re(z)t}+e^{-\Re(z)t}}{e^{\pi t}-e^{-\pi t}}
  \left|\frac{\psi(\gamma t)}{\gamma t}\right|\,dt.
\end{split}
\end{equation*}
Since $\psi(w)/w$ is continuous in $D_0(\pi)$ and $\lim_{w\to0}\psi(w)/w=1/6$, there exists $\delta>0$ such that $\min_{|w|<\delta}|\psi(w)/w|=D>0$.
Therefore if $a<\delta/|\gamma|$, then $|\gamma t|\le|\gamma|a<\delta$ in the integral and so we have
\begin{equation*}
\begin{split}
  &\left|
    \int_{-\delta/|\gamma|}^{-R}
    \frac{e^{zt}}{t\sinh(\pi t)}\psi(\gamma t)\,dt
    +
    \int_{R}^{\delta/|\gamma|}
    \frac{e^{zt}}{t\sinh(\pi t)}\psi(\gamma t)\,dt
  \right|
  \\
  \le&
  2|\gamma|D
  \int_{R}^{\delta/|\gamma|}
  \frac{e^{\Re(z)t}+e^{-\Re(z)t}}{e^{\pi t}-e^{-\pi t}}
  \,dt
  \\
  \le&
  \frac{2|\gamma|D}{1-e^{-2\pi R}}
  \int_{R}^{a}(e^{-(\pi-\Re(z))t}+e^{-(\pi+\Re(z))t})
  \,dt.
\end{split}
\end{equation*}
Then from the argument in \cite[Page~533]{Andersen/Hansen:JKNOT2006} we have
\begin{equation*}
\begin{split}
  &\left|
    \int_{-\delta/|\gamma|}^{-R}
    \frac{e^{zt}}{t\sinh(\pi t)}\psi(\gamma t)\,dt
    +
    \int_{R}^{\delta/|\gamma|}
    \frac{e^{zt}}{t\sinh(\pi t)}\psi(\gamma t)\,dt
  \right|
  \\
  \le&
  \frac{4\delta D}{1-e^{-2\pi R}}
\end{split}
\end{equation*}
when $|\Re(z)|\le\pi$, and
\begin{equation*}
\begin{split}
  &\left|
    \int_{-\delta/|\gamma|}^{-R}
    \frac{e^{zt}}{t\sinh(\pi t)}\psi(\gamma t)\,dt
    +
    \int_{R}^{\delta/|\gamma|}
    \frac{e^{zt}}{t\sinh(\pi t)}\psi(\gamma t)\,dt
  \right|
  \\
  \le&
  \frac{2|\gamma|D}{(1-e^{-2\pi R})}
  \left(
    \frac{1-e^{-(\pi-\Re(z))\delta/|\gamma|}}{\pi-\Re(z)}
    +
    \frac{1-e^{-(\pi+\Re(z))\delta/|\gamma|}}{\pi+\Re(z)}
  \right)
\end{split}
\end{equation*}
when $|\Re(z)|<\pi$.
\par
Therefore if $|\Re(z)|\le\pi$, we have
\begin{equation*}
\begin{split}
  I_{\gamma}(z)
  &\le
  \frac{E}{\delta\left(1-e^{-2\pi|\delta|/|\gamma|}\right)}
  +
  |\gamma|B\left(1+e^{\Im(z)R}\right)
  +
  \frac{\delta D}{1-e^{-2\pi R}}
  \\
  &\le
  \frac{1}{1-e^{-2\pi R}}\left(\frac{E}{\delta}+\delta D\right)
  +
  |\gamma|B\left(1+e^{\Im(z)R}\right)
\end{split}
\end{equation*}
since $\delta/|\gamma|\ge R$.
If $|\Re(z)|<\pi$ we also have
\begin{equation*}
\begin{split}
  &|I_{\gamma}(z)|
  \\
  \le&
  \frac{E|\gamma|}
       {2\delta^2\left(1-e^{-2\pi\delta/|\gamma|}\right)}
  \left(
    \frac{e^{-(\pi-\Re(z))\delta/|\gamma|)}}{\pi-\Re(z)}
    +
    \frac{e^{-(\pi+\Re(z))\delta/|\gamma|)}}{\pi+\Re(z)}
  \right)
  \\
  &+
  \frac{|\gamma|D}{2(1-e^{-2\pi R})}
  \left(
    \frac{1-e^{-(\pi-\Re(z))\delta/|\gamma|}}{\pi-\Re(z)}
    +
    \frac{1-e^{-(\pi+\Re(z))\delta/|\gamma|}}{\pi+\Re(z)}
  \right)
  \\
  &+
  |\gamma|B\left(1+e^{\Im(z)R}\right)
  \\
  \le&
  |\gamma|
  \left(
    \frac{1}{\pi-\Re(z)}+\frac{1}{\pi+\Re(z)}
  \right)
  \frac{1}{1-e^{-2\pi R}}
  \left(\frac{E}{2\delta^2}+\frac{D}{2}\right)+
  |\gamma|B\left(1+e^{\Im(z)R}\right).
\end{split}
\end{equation*}
\par
The lemma follows by putting
\begin{equation*}
  A
  :=
  \frac{1}{1-e^{-2\pi R}}
  \times
  \max
  \left\{
    \left(\frac{E}{2\delta}+\frac{\delta D}{2}\right),
    \left(\frac{E}{2\delta^2}+\frac{D}{2}\right)
  \right\}.
\end{equation*}
\end{proof}
%%%%%%%%%%%%%%%%%%%%%%%%%%%%%%%%%%%%%%%%%%%%%%%%%%%%%%%%%%%%%%%%%%%%%%%%%%%
Now we prove Proposition~\ref{prop:4.7}.
\par
First note that since $g_N(x)$ has no poles inside $C_{+}(\varepsilon)\cup C_{-}(\varepsilon)$, we can assume that $\varepsilon=0$ without changing the sum, that is,
\begin{equation*}
\begin{split}
  &\int_{C_{+}(\varepsilon)}(\tan(N\pi x)-\sqrt{-1})g_N(x)\,dx
  +
  \int_{C_{-}(\varepsilon)}(\tan(N\pi x)+\sqrt{-1})g_N(x)\,dx
  \\
  =&
  \int_{C_{+}(0)}(\tan(N\pi x)-\sqrt{-1})g_N(x)\,dx
  +
  \int_{C_{-}(0)}(\tan(N\pi x)+\sqrt{-1})g_N(x)\,dx
\end{split}
\end{equation*}
We decompose $C_{+}(0)$ into $C_{+,1}\cup C_{+,2}\cup C_{+,3}$, where $C_{+,1}$ connects $0$ and $-u/(2\pi)+\sqrt{-1}$, $C_{+,2}$ connects $-u/(2\pi)+\sqrt{-1}$ and $1-u/(2\pi)+\sqrt{-1}$, and $C_{+,3}$ connects $1-u/(2\pi)+\sqrt{-1}$ and $1$.
Similarly we also decompose $C_{-}(0)$ in to $C_{-,1}\cup C_{-,2}\cup C_{-,3}$, where $C_{-,1}$ connects $0$ and $u/(2\pi)-\sqrt{-1}$, $C_{-,2}$ connects $u/(2\pi)-\sqrt{-1}$ and $1+u/(2\pi)-\sqrt{-1}$, and $C_{-,3}$ connects $1+u/(2\pi)-\sqrt{-1}$ and $-1$.
Let $I_{\pm,i}(N)$ be the integral of $\bigl(\tan(N\pi x)-\sqrt{-1}\bigr)g_N(x)$ along $C_{\pm,i}$ ($i=1,2,3$).
We will show that $|I_{\pm,i}(N)|$ is bounded from above by $K_{\pm,i}/N$ for a positive constant $K_{\pm,i}$ independent of $N$.
\par
We will give the following estimations for $I_{\pm,i}(N)$.
%%%%%%%%%%%%%%%%%%%%%%%%%%%%%%%%%%%%%%%%%%%%%%%%%%%%%%%%%%%%%%%%%%%%%%%
\begin{equation}\label{eq:I_{+,1}}
  |I_{+,1}(N)|<\frac{K_{+,1}}{N}.
\end{equation}
\begin{equation}\label{eq:I_{+,2}}
  |I_{+,2}(N)|<\frac{K_{+,2}}{N}.
\end{equation}
\begin{equation}\label{eq:I_{+,3}}
  |I_{+,3}(N)|<\frac{K_{+,3}}{N}.
\end{equation}
\begin{equation}\label{eq:I_{-,1}}
  |I_{-,1}(N)|<\frac{K_{-,1}}{N}.
\end{equation}
\begin{equation}\label{eq:I_{-,2}}
  |I_{-,2}(N)|<\frac{K_{-,2}}{N}.
\end{equation}
\begin{equation}\label{eq:I_{-,3}}
  |I_{-,3}(N)|<\frac{K_{-,3}}{N}.
\end{equation}
%%%%%%%%%%%%%%%%%%%%%%%%%%%%%%%%%%%%%%%%%%%%%%%%%%%%%%%%%%%%%%%
\begin{proof}[Proof of \eqref{eq:I_{+,1}}]
We first estimate $|\tan(N\pi (-u/(2\pi)+\sqrt{-1})t)-\sqrt{-1}|$.
We have
\begin{equation*}
\begin{split}
  &
  |\tan(N\pi (-u/(2\pi)+\sqrt{-1})t))-\sqrt{-1}|
  \\
  =&
  \left|
    \frac{2e^{-Nut\sqrt{-1}/2-N\pi t)}}
         {e^{-Nut\sqrt{-1}/2-N\pi t)}+e^{Nut\sqrt{-1}/2+N\pi t)}}
  \right|
  \\
  =&
  \frac{2e^{-2N\pi t}}
  {\left|e^{-Nut\sqrt{-1}-2N\pi t}+1\right|}.
\end{split}
\end{equation*}
Since the denominator is bigger than $1$ if $Nut<\pi/2$ we have
\begin{equation*}
  |\tan(N\pi (-u/(2\pi)+\sqrt{-1})t))-\sqrt{-1}|
  <
  2e^{-2N\pi t}
\end{equation*}
when $t<\pi/(2Nu)$.
If $t\ge\pi/(2Nu)$, then the denominator is bigger than equal to $1-e^{-2N\pi t}$.
So we have
\begin{equation*}
  |\tan(N\pi (-u/(2\pi)+\sqrt{-1})t))-\sqrt{-1}|
  \le
  \frac{2e^{-2N\pi t}}{1-e^{-2N\pi t}}
  \le
  \frac{2e^{-2N\pi t}}{1-e^{-\pi^2/u}}
\end{equation*}
when $t\ge\pi/(2Nu)$.
Therefore for any $0\le t\le 1$ we have
\begin{equation}\label{eq:I_{+,1}tan}
  |\tan(N\pi (-u/(2\pi)+\sqrt{-1})t)-\sqrt{-1}|
  <
  \frac{2e^{-2N\pi t}}{1-e^{-\pi^2/u}}.
\end{equation}
\par
So we have
\begin{equation}\label{eq:I_{+,1}integral}
  \left|I_{+,1}(N)\right|
  \le
  \frac{2}{1-e^{-\pi^2/u}}
  \int_{0}^{1}
  e^{-2N\pi t}
  \left|
    g_N\left(\left(-\frac{u}{2\pi}+\sqrt{-1}\right)t\right)
  \right|
  \,dt.
\end{equation}
\par
We estimate $|g_N((-u/(2\pi)+\sqrt{-1})t)|$.
\par
Since the function $g_N$ is not well-defined on the segment $(-u/(2\pi)+\sqrt{-1})t$ (Figure~\ref{fig:contour}), we need to consider the segment $(-u/(2\pi)+\sqrt{-1})t+\varepsilon)$ ($0\le t\le1$) instead for small $\varepsilon$.
(See the argument in \cite[Page~534]{Andersen/Hansen:JKNOT2006}.)
\par
From \eqref{eq:g_N} we have
\begin{equation*}
\begin{split}
  &g_{N}\left(\left(-\frac{u}{2\pi}+\sqrt{-1}\right)t+\varepsilon\right)
  \\
  =&
  \exp
  \left[
    N\Phi\left(\left(-\frac{u}{2\pi}+\sqrt{-1}\right)t+\varepsilon\right)
  \right]
  \\
  &
  \times
  \exp
  \left[
    I_{\gamma}
    \left(
      \pi-\sqrt{-1}u-\sqrt{-1}\frac{|\xi|^2t}{2\pi}+\sqrt{-1}\varepsilon\xi
    \right)
  \right.
  \\
  &\phantom{\times\exp\bigl[}
  \left.
    -
    I_{\gamma}
    \left(
      -\pi-\sqrt{-1}u+\sqrt{-1}\frac{|\xi|^2t}{2\pi}-\sqrt{-1}\varepsilon\xi
    \right)
  \right].
\end{split}
\end{equation*}
\par
From Lemma~\ref{lem:Lemma3}, there exist $A>0$ and $B>0$ such that $|I_{\gamma}(z)|\le 2A+B|\gamma|(1+e^{-\Im(z)R})$.
So we have
\begin{equation}\label{eq:I_{+,1}Phi}
\begin{split}
  &
  \left|
    g_{N}\left(\left(-\frac{u}{2\pi}+\sqrt{-1}\right)t+\varepsilon\right)
  \right|
  \\
  \le&
  \exp
  \left[
    N\Re\Phi\left(\left(-\frac{u}{2\pi}+\sqrt{-1}\right)t+\varepsilon\right)
  \right]
  \frac{\exp\left(2A+B|\gamma|(1+e^{(u+|\xi|^2t/(2\pi)-\varepsilon u)R}\right)}
       {\exp\left(2A+B|\gamma|(1+e^{(u-|\xi|^2t/(2\pi)+\varepsilon u)R}\right)}
  \\
  \le&
  \exp
  \left[
    N\Re\Phi\left(\left(-\frac{ut}{2\pi}+\sqrt{-1}\right)t+\varepsilon\right)
  \right]
  \exp\left(2A+B|\gamma|(1+e^{(u+|\xi|^2/(2\pi)-\varepsilon u)R}\right)
\end{split}
\end{equation}
for $0\le t\le1$.
\par
Now we want to estimate $\Re\Phi\bigl((-ut/(2\pi)+\sqrt{-1})t+\varepsilon\bigr)$.
\par
From the definition we have
\begin{multline*}
  \Phi\left(\left(-\frac{u}{2\pi}+\sqrt{-1}\right)t+\varepsilon\right)
  \\
  =
  \frac{1}{\xi}
  \left(
    \Li_2(e^{u+|\xi|^2t/(2\pi)-\varepsilon\xi})
    -
    \Li_2(e^{u-|\xi|^2t/(2\pi)+\varepsilon\xi})
  \right)
  +\frac{u^2t}{2\pi}-\sqrt{-1}ut-\varepsilon u.
\end{multline*}
Since we may assume that $\Re(u+|\xi|^2t/(2\pi)-\xi\varepsilon)=(1-\varepsilon)u+|\xi|^2t/(2\pi)>0$, there are two cases to consider; the case where $u-|\xi|^2t/(2\pi)<0$ and the case where $u-|\xi|^2t/(2\pi)\ge0$.
\par
If $|z|>1$, it is convenient to replace $\Li_2(z)$ with $\Li_2(z^{-1})$ using the following well-known formula.
\begin{equation}\label{eq:dilog}
  \Li_2(z)+\Li_2(z^{-1})
  =
  -\frac{\pi^2}{6}-\frac{1}{2}\bigl(\log(-z)\bigr)^2,
\end{equation}
where we choose a branch of $\log(-z)$ so that $-\pi<\Im\log(-z)<\pi$.
\begin{itemize}
\item The case where $u-|\xi|^2t/(2\pi)<0$.
We choose $\varepsilon$ small enough so that $u-|\xi|^2t/(2\pi)+\varepsilon u<0$.
From \eqref{eq:dilog} we have
\begin{equation*}
\begin{split}
  &\Phi\left(\left(-\frac{u}{2\pi}+\sqrt{-1}\right)t+\varepsilon\right)
  \\
  =&
  \frac{1}{\xi}
  \left(\vphantom{
          \left(u+\frac{|\xi|^2t}{2\pi}-\varepsilon\xi+\sqrt{-1}\pi\right)^2}
    -
    \Li_2(e^{-u-|\xi|^2t/(2\pi)+\varepsilon\xi})
    -
    \Li_2(e^{u-|\xi|^2t/(2\pi)+\varepsilon\xi})
  \right.
  \\
  &
  \left.\quad
    -\frac{\pi^2}{6}
    -\frac{1}{2}
    \left(
      u+\frac{|\xi|^2t}{2\pi}-\varepsilon\xi+\sqrt{-1}\pi
    \right)^2
  \right)
  \\
  &+\frac{u^2t}{2\pi}-\sqrt{-1}ut-\varepsilon u
  \\
  =&
  \frac{1}{\xi}
  \left(
    -
    \Li_2(e^{u-|\xi|^2t/(2\pi)+\xi\varepsilon})
    -
    \Li_2(e^{u-|\xi|^2t/(2\pi)+\xi\varepsilon})
  \right)
  \\
  &-\frac{\pi^2}{6\xi}
  -\frac{1}{2\xi}
  \left(
    u+\frac{|\xi|^2t}{2\pi}-\varepsilon\xi+\sqrt{-1}\pi
  \right)^2
  \\
  &+\frac{u^2t}{2\pi}-\sqrt{-1}ut-\varepsilon u,
\end{split}
\end{equation*}
where in the first equality we choose the sign of $\sqrt{-1}\pi$ so that $\Im(u+|\xi|^2t/(2\pi)-\varepsilon\xi+\pi\sqrt{-1})=-\varepsilon u+\pi$ is between $-\pi$ and $\pi$.
Since the dilogarithm function $\Li(z)$ is analytic when $\Re(z)<1$, we have
\begin{equation*}
\begin{split}
  &\lim_{\varepsilon\searrow0}
  \Phi\left(\left(-\frac{u}{2\pi}+\sqrt{-1}\right)t+\varepsilon\right)
  \\
  =&
  \frac{1}{\xi}
  \left(
    -
    \Li_2(e^{-u-|\xi|^2t/(2\pi)})
    -
    \Li_2(e^{u-|\xi|^2t/(2\pi)})
  \right)
  \\
  &-\frac{\pi^2}{6\xi}
  -\frac{1}{2\xi}
  \left(
    u+\frac{|\xi|^2t}{2\pi}+\sqrt{-1}\pi
  \right)^2
  +\frac{u^2t}{2\pi}-\sqrt{-1}ut.
\end{split}
\end{equation*}
Since $\Li_2(z)$ is real when $z$ is real and $z<1$, we have
\begin{equation*}
\begin{split}
  &\lim_{\varepsilon\searrow0}
  \Re\Phi\left(\left(-\frac{u}{2\pi}+\sqrt{-1}\right)t+\varepsilon\right)
  \\
  =&
  -\frac{u}{|\xi|^2}
  \left(
    \Li_2(e^{-u-|\xi|^2t/(2\pi)})
    +
    \Li_2(e^{u-|\xi|^2t/(2\pi)})
  \right)
  \\
  &-\frac{u\pi^2}{6|\xi|^2}
  -\frac{u}{2|\xi|^2}
  \left(
    u+\frac{|\xi|^2t}{2\pi}
  \right)^2
  +
  \frac{u\pi^2}{2|\xi|^2}
  -
  \frac{4\pi(u+|\xi|^2t/(2\pi))\pi}{2|\xi|^2}
  +\frac{u^2t}{2\pi}
  \\
  =&
  -\frac{u}{|\xi|^2}
  \left(
    \Li_2(e^{-u-|\xi|^2t/(2\pi)})
    +
    \Li_2(e^{u-|\xi|^2t/(2\pi)})
  \right)
  \\
  &-\frac{5u\pi^2}{6|\xi|^2}
  -\frac{u^3}{2|\xi|^2}
  -\frac{u|\xi|^2t^2}{8\pi^2}
  -\pi t
  \\
  <&0
\end{split}
\end{equation*}
if $u>0$.
\item The case where $u-|\xi|^2t/(2\pi)\ge0$.
In this case we have
\begin{equation*}
\begin{split}
  &\Phi\left(\left(-\frac{u}{2\pi}+\sqrt{-1}\right)t+\varepsilon\right)
  \\
  =&
  \frac{1}{\xi}
  \left(
    \vphantom{\left(\frac{|\xi|^2t}{2\pi}\right)^2}
    -
    \Li_2(e^{-u-|\xi|^2t/(2\pi)+\varepsilon\xi})
    +
    \Li_2(e^{-u+|\xi|^2t/(2\pi)-\varepsilon\xi})
  \right.
  \\
  &\left.\quad
    -\frac{1}{2}
    \left(
      u+\frac{|\xi|^2t}{2\pi}-\varepsilon\xi+\pi\sqrt{-1}
    \right)^2
    +\frac{1}{2}
    \left(
      u-\frac{|\xi|^2t}{2\pi}+\varepsilon\xi-\pi\sqrt{-1}
    \right)^2
  \right)
  \\
  &+
  \frac{u^2t}{2\pi}-\sqrt{-1}ut-u\varepsilon
  \\
  =&
  \frac{1}{\xi}
  \left(
    \vphantom{\left(\frac{|\xi|^2t}{2\pi}\right)^2}
    -
    \Li_2(e^{-u-|\xi|^2t/(2\pi)+\varepsilon\xi})
    +
    \Li_2(e^{-u+|\xi|^2t/(2\pi)-\varepsilon\xi})
  \right)
  \\
  &
  -\frac{2u}{\xi}
  \left(
    \frac{|\xi|^2t}{2\pi}-\varepsilon\xi+\pi\sqrt{-1}
  \right)
  +
  \frac{u^2t}{2\pi}-\sqrt{-1}ut-u\varepsilon
\end{split}
\end{equation*}
and so we have
\begin{equation*}
\begin{split}
  &\lim_{\varepsilon\searrow0}
  \Re\Phi\left(\left(-\frac{u}{2\pi}+\sqrt{-1}\right)t+\varepsilon\right)
  \\
  =&
  \frac{u}{|\xi|^2}
  \left(
    -
    \Li_2(e^{-u-|\xi|^2t/(2\pi)})
    +
    \Li_2(e^{-u+|\xi|^2t/(2\pi)})
  \right)
  \\
  &
  -\Re
  \left(
    \frac{2u}{\xi}
    \left(
      \frac{|\xi|^2t}{2\pi}+\pi\sqrt{-1}
    \right)
  \right)
  +
  \frac{u^2t}{2\pi}  \\
  =&
  \frac{u}{|\xi|^2}
  \left(
    -
    \Li_2(e^{-u-|\xi|^2t/(2\pi)})
    +
    \Li_2(e^{-u+|\xi|^2t/(2\pi)})
  \right)
  -\frac{4\pi^2u}{|\xi|^2}-\frac{u^2t}{2\pi}
  \\
  <&0
\end{split}
\end{equation*}
if $u>0$.
\end{itemize}
Therefore for any $t$ we have $\Re\Phi\bigl((-u/(2\pi)+\sqrt{-1})t+\varepsilon\bigr)\le0$ for small $\varepsilon>0$.
\par
So from \eqref{eq:I_{+,1}integral} and \eqref{eq:I_{+,1}Phi} we have
\begin{equation*}
\begin{split}
  \left|I_{+,1}(N)\right|
  &\le
  \frac{2}{1-e^{-\pi^2/u}}
  \exp\left(2A+B|\gamma|(1+e^{(u+|\xi|^2/(2\pi)-\varepsilon u)R}\right)
  \int_{0}^{1}e^{-2N\pi t}\,dt
  \\
  &=
  \frac{1-e^{-2N\pi}}{N\pi(1-e^{-\pi^2/u})}
  \exp\left(2A+B|\gamma|(1+e^{(u+|\xi|^2/(2\pi))R}\right)
  \\
  &<
  \frac{K_{+,1}}{N}
\end{split}
\end{equation*}
for a positive constant $K_{+,1}$.
\end{proof}%{of I_{+,1}}
%%%%%%%%%%%%%%%%%%%%%%%%%%%%%%%%%%%%%%%%%%%%%%%%%%%%%%%%%%%%%%%
\begin{proof}[Proof of \eqref{eq:I_{-,1}}]
Since $\tan$ is an odd function we have from \eqref{eq:I_{+,1}tan}
\begin{equation}\label{eq:I_{-,1}tan}
  |\tan\bigl(N\pi(u/(2\pi)-\sqrt{-1})t\bigr)+\sqrt{-1}|
  <
  \frac{2e^{-2N\pi t}}{1-e^{-\pi^2/u}}
\end{equation}
for any $0\le t\le 1$.
\par
So we have
\begin{equation}\label{eq:I_{-,1}integral}
  \left|I_{-,1}(N)\right|
  \le
  \frac{2}{1-e^{-\pi^2/u}}
  \int_{0}^{1}
  e^{-2N\pi t}
  \left|
    g_N\left(\left(\frac{u}{2\pi}-\sqrt{-1}\right)t\right)
  \right|
  \,dt.
\end{equation}
\par
We estimate $|g_N((u/(2\pi)-\sqrt{-1})t)|$.
\par
As in the case of $I_{+,1}$ we need to calculate $g_N(g_N((u/(2\pi)-\sqrt{-1})t)+\varepsilon)$.
\par
From \eqref{eq:g_N} we have
\begin{equation*}
\begin{split}
  &g_{N}\left(\left(\frac{u}{2\pi}-\sqrt{-1}\right)t+\varepsilon\right)
  \\
  =&
  \exp
  \left[
    N\Phi\left(\left(\frac{u}{2\pi}-\sqrt{-1}\right)t+\varepsilon\right)
  \right]
  \\
  &
  \times
  \exp
  \left[
    I_{\gamma}
    \left(
      \pi-\sqrt{-1}u+\sqrt{-1}\frac{|\xi|^2t}{2\pi}+\sqrt{-1}\varepsilon\xi
    \right)
  \right.
  \\
  &\phantom{\times\exp\bigl[}
  \left.
    -
    I_{\gamma}
    \left(
      -\pi-\sqrt{-1}u-\sqrt{-1}\frac{|\xi|^2t}{2\pi}-\sqrt{-1}\varepsilon\xi
    \right)
  \right].
\end{split}
\end{equation*}
\par
From Lemma~\ref{lem:Lemma3}, there exist $A>0$ and $B>0$ such that $|I_{\gamma}(z)|\le 2A+B|\gamma|(1+e^{-\Im(z)R})$.
So we have
\begin{equation}\label{eq:I_{-,1}Phi}
\begin{split}
  &
  \left|
    g_{N}\left(\left(\frac{u}{2\pi}-\sqrt{-1}\right)t+\varepsilon\right)
  \right|
  \\
  \le&
  \exp
  \left[
    N\Re\Phi\left(\left(\frac{u}{2\pi}-\sqrt{-1}\right)t+\varepsilon\right)
  \right]
  \frac{\exp\left(2A+B|\gamma|(1+e^{(u-|\xi|^2t/(2\pi)-\varepsilon u)R}\right)}
       {\exp\left(2A+B|\gamma|(1+e^{(u+|\xi|^2t/(2\pi)+\varepsilon u)R}\right)}
  \\
  \le&
  \exp
  \left[
    N\Re\Phi\left(\left(\frac{ut}{2\pi}-\sqrt{-1}\right)t+\varepsilon\right)
  \right]
  \exp\left(2A+B|\gamma|(1+e^{(u-|\xi|^2/(2\pi)-\varepsilon u)R}\right)
\end{split}
\end{equation}
for $0\le t\le1$.
\par
Now we want to estimate $\Re\Phi\bigl((ut/(2\pi)-\sqrt{-1})t+\varepsilon\bigr)$.
\par
From the definition we have
\begin{multline*}
  \Phi\left(\left(\frac{u}{2\pi}-\sqrt{-1}\right)t+\varepsilon\right)
  \\
  =
  \frac{1}{\xi}
  \left(
    \Li_2(e^{u-|\xi|^2t/(2\pi)-\varepsilon\xi})
    -
    \Li_2(e^{u+|\xi|^2t/(2\pi)+\varepsilon\xi})
  \right)
  -\frac{u^2t}{2\pi}+\sqrt{-1}ut-\varepsilon u.
\end{multline*}
Since $\Re(u+|\xi|^2t/(2\pi)+\varepsilon\xi)=(1+\varepsilon)u+|\xi|^2t/(2\pi)>0$, there are two cases to consider; the case where $u-|\xi|^2t/(2\pi)<0$ and the case where $u-|\xi|^2t/(2\pi)\ge0$.
\begin{itemize}
\item The case where $u-|\xi|^2t/(2\pi)\le0$.
In this case we have $u-|\xi|^2t/(2\pi)-\varepsilon u<0$.
From \eqref{eq:dilog} we have
\begin{equation*}
\begin{split}
  &\Phi\left(\left(\frac{u}{2\pi}-\sqrt{-1}\right)t+\varepsilon\right)
  \\
  =&
  \frac{1}{\xi}
  \left(\vphantom{
          \left(u+\frac{|\xi|^2t}{2\pi}-\varepsilon\xi+\sqrt{-1}\pi\right)^2}
    \Li_2(e^{u-|\xi|^2t/(2\pi)-\varepsilon\xi})
    +
    \Li_2(e^{-u-|\xi|^2t/(2\pi)-\varepsilon\xi})
  \right.
  \\
  &
  \left.\quad
    +\frac{\pi^2}{6}
    +\frac{1}{2}
    \left(
      u+\frac{|\xi|^2t}{2\pi}+\varepsilon\xi-\sqrt{-1}\pi
    \right)^2
  \right)
  \\
  &-\frac{u^2t}{2\pi}+\sqrt{-1}ut-\varepsilon u
  \\
  =&
  \frac{1}{\xi}
  \left(
    \Li_2(e^{u-|\xi|^2t/(2\pi)-\xi\varepsilon})
    +
    \Li_2(e^{-u-|\xi|^2t/(2\pi)-\xi\varepsilon})
  \right)
  \\
  &+\frac{\pi^2}{6\xi}
  +\frac{1}{2\xi}
  \left(
    u+\frac{|\xi|^2t}{2\pi}+\varepsilon\xi-\sqrt{-1}\pi
  \right)^2
  \\
  &-\frac{u^2t}{2\pi}+\sqrt{-1}ut-\varepsilon u.
\end{split}
\end{equation*}
Therefore we have
\begin{equation*}
\begin{split}
  &\lim_{\varepsilon\searrow0}
  \Phi\left(\left(\frac{u}{2\pi}-\sqrt{-1}\right)t+\varepsilon\right)
  \\
  =&
  \frac{1}{\xi}
  \left(
    \Li_2(e^{u-|\xi|^2t/(2\pi)})
    +
    \Li_2(e^{-u-|\xi|^2t/(2\pi)})
  \right)
  \\
  &+\frac{\pi^2}{6\xi}
  +\frac{1}{2\xi}
  \left(
    u+\frac{|\xi|^2t}{2\pi}-\sqrt{-1}\pi
  \right)^2
  \\
  &-\frac{u^2t}{2\pi}+\sqrt{-1}ut
\end{split}
\end{equation*}
and so
\begin{equation*}
\begin{split}
  &\lim_{\varepsilon\searrow0}
  \Re\Phi\left(\left(\frac{u}{2\pi}-\sqrt{-1}\right)t+\varepsilon\right)
  \\
  =&
  \frac{u}{|\xi|^2}
  \left(
    \Li_2(e^{u-|\xi|^2t/(2\pi)})
    +
    \Li_2(e^{-u-|\xi|^2t/(2\pi)})
  \right)
  \\
  &+\frac{u\pi^2}{6|\xi|^2}
  +\frac{u}{2|\xi|^2}
  \left(
    u+\frac{|\xi|^2t}{2\pi}
  \right)^2
  -
  \frac{u\pi^2}{2|\xi|^2}
  -
  \frac{4\pi(u+|\xi|^2t/(2\pi))\pi}{2|\xi|^2}
  -\frac{u^2t}{2\pi}
  \\
  =&
  \frac{u}{|\xi|^2}
  \left(
    \Li_2(e^{u-|\xi|^2t/(2\pi)})
    +
    \Li_2(e^{-u-|\xi|^2t/(2\pi)})
  \right)
  \\
  &-\frac{7u\pi^2}{3|\xi|^2}
  +\frac{u^3}{2|\xi|^2}
  +\frac{u|\xi|^2t^2}{8\pi^2}
  -\pi t
  \\
  =&
  -2\frac{u\pi^2}{|\xi|^2}
  +\frac{u^3}{2|\xi|^2}
  +\frac{u|\xi|^2t^2}{8\pi^2}
  -\pi t
\end{split}
\end{equation*}
since $\Li_2(z)\le\pi^2/6$ for $0<z\le1$.
We can easily prove
\begin{equation*}
  -2\frac{u\pi^2}{|\xi|^2}
  +\frac{u^3}{2|\xi|^2}
  +\frac{u|\xi|^2t^2}{8\pi^2}
  -\pi t
  <
  \pi t
\end{equation*}
when $0<u<1$ and $2u\pi/|\xi|^2\le t\le1$.
So we have
\begin{equation*}
  \lim_{\varepsilon\searrow0}
  \Re\Phi\left(\left(\frac{u}{2\pi}-\sqrt{-1}\right)t+\varepsilon\right)
  <\pi t
\end{equation*}
in this case.
\item The case where $u-|\xi|^2t/(2\pi)>0$.
We choose $\varepsilon$ so that $u-|\xi|^2t/(2\pi)-\varepsilon u>0$.
In this case we have
\begin{equation*}
\begin{split}
  &\Phi\left(\left(\frac{u}{2\pi}-\sqrt{-1}\right)t+\varepsilon\right)
  \\
  =&
  \frac{1}{\xi}
  \left(
    \vphantom{\left(\frac{|\xi|^2t}{2\pi}\right)^2}
    -
    \Li_2(e^{-u+|\xi|^2t/(2\pi)+\varepsilon\xi})
    +
    \Li_2(e^{-u-|\xi|^2t/(2\pi)-\varepsilon\xi})
  \right.
  \\
  &\left.\quad
    -\frac{1}{2}
    \left(
      u-\frac{|\xi|^2t}{2\pi}-\varepsilon\xi+\pi\sqrt{-1}
    \right)^2
    +\frac{1}{2}
    \left(
      -u-\frac{|\xi|^2t}{2\pi}-\varepsilon\xi+\pi\sqrt{-1}
    \right)^2
  \right)
  \\
  &-
  \frac{u^2t}{2\pi}+\sqrt{-1}ut-\varepsilon u
\end{split}
\end{equation*}
and so we have
\begin{equation*}
\begin{split}
  &\lim_{\varepsilon\searrow0}
  \Re\Phi\left(\left(\frac{u}{2\pi}-\sqrt{-1}\right)t+\varepsilon\right)
  \\
  =&
  \frac{u}{|\xi|^2}
  \left(
    -
    \Li_2(e^{-u+|\xi|^2t/(2\pi)})
    +
    \Li_2(e^{-u-|\xi|^2t/(2\pi)})
  \right)
  +\frac{u^2t}{2\pi}-\frac{4u\pi^2}{|\xi|^2}
  \\
  <&
  \frac{u^2t}{2\pi}-\frac{23u\pi^2}{6|\xi|^2}
  <0
\end{split}
\end{equation*}
if $0<u<1$ and $0\le t<2u\pi/|\xi|^2$.
\end{itemize}
Therefore for any $0\le t\le1$ we have $\Re\Phi\bigl((u/(2\pi)-\sqrt{-1})t+\varepsilon\bigr)\le\pi t$ for small $\varepsilon>0$.
\par
So from \eqref{eq:I_{-,1}integral} and \eqref{eq:I_{-,1}Phi} we have
\begin{equation*}
\begin{split}
  \left|I_{-,1}(N)\right|
  &\le
  \frac{2}{1-e^{-\pi^2/u}}
  \exp\left(2A+B|\gamma|(1+e^{(u-|\xi|^2/(2\pi))R}\right)
  \int_{0}^{1}e^{-N\pi t}\,dt
  \\
  &=
  \frac{2(1-e^{-N\pi})}{N(1-e^{-\pi^2/u})}
  \exp\left(2A+B|\gamma|(1+e^{(u-|\xi|^2/(2\pi))R}\right)
  \\
  &<
  \frac{K_{-,1}}{N}
\end{split}
\end{equation*}
for a positive constant $K_{-,1}$.
\end{proof}%{of I_{-,1}}
%%%%%%%%%%%%%%%%%%%%%%%%%%%%%%%%%%%%%%%%%%%%%%%%%%%%%%%%%%%
\begin{proof}[Proof of \eqref{eq:I_{+,3}}]
Since $\tan$ has period $\pi$ we have
\begin{equation}\label{eq:tan_I+3}
  |\tan(N\pi (-u/(2\pi)+\sqrt{-1})t)-\sqrt{-1}|
  <
  \frac{2e^{-2N\pi t}}{1-e^{-\pi^2/u}}
\end{equation}
from \eqref{eq:I_{+,1}tan}
\par
From \eqref{eq:tan_I+3} we have
\begin{equation}\label{eq:I_{+,3}integral}
  \left|I_{+,3}(N)\right|
  \le
  \frac{2}{1-e^{-\pi^2/u}}
  \int_{0}^{1}
  e^{-2N\pi t}
  \left|
    g_N\left(\left(-\frac{u}{2\pi}+\sqrt{-1}\right)t+1\right)
  \right|
  \,dt.
\end{equation}
\par
As in the case of $I_{+,1}(N)$ we consider the integral on the segment $(-u/(2\pi)+\sqrt{-1})t+1-\varepsilon$ ($0\le t\le1$) for small $\varepsilon$.
\par
We estimate $\Phi\bigl((-u/(2\pi)+\sqrt{-1})t+1-\varepsilon\bigr)$.
We have
\begin{equation*}
\begin{split}
  &\Phi\left(\left(-\frac{u}{2\pi}+\sqrt{-1}\right)t+1-\varepsilon\right)
  \\
  =&
  \frac{1}{\xi}
  \left(
    \Li_2(e^{u+|\xi|^2t/(2\pi)-(1-\varepsilon)\xi})
    -
    \Li_2(e^{u-|\xi|^2t/(2\pi)+(1-\varepsilon)\xi})
  \right)
  \\
  &+\frac{u^2t}{2\pi}-\sqrt{-1}ut-(1-\varepsilon)u.
\end{split}
\end{equation*}
Since we may assume that $\Re(u+|\xi|^2t/(2\pi)-(1-\varepsilon)\xi)=\varepsilon u+|\xi|^2t/(2\pi)>0$, there are two cases to consider; the case where $2u-|\xi|^2t/(2\pi)\le0$ and the case where $2u-|\xi|^2t/(2\pi)>0$.
\begin{itemize}
\item $2u-|\xi|^2t/(2\pi)\le0$.
In this case, since $u-|\xi|^2t/(2\pi)+(1-\varepsilon)u<0$, from \eqref{eq:dilog} we have
\begin{equation*}
\begin{split}
  &\Phi\left(\left(-\frac{u}{2\pi}+\sqrt{-1}\right)t+1-\varepsilon\right)
  \\
  =&
  \frac{1}{\xi}
  \left(
  \vphantom{
    \left(
      u+\frac{|\xi|^2t}{2\pi}-(1-\varepsilon)\xi+\pi\sqrt{-1}
    \right)^2}
    -
    \Li_2(e^{-u-|\xi|^2t/(2\pi)+(1-\varepsilon)\xi})
    -
    \Li_2(e^{u-|\xi|^2t/(2\pi)+(1-\varepsilon)\xi})
  \right.
  \\
  &
  \quad
  \left.
    -\frac{\pi^2}{6}
    -\frac{1}{2}
    \left(
      u+\frac{|\xi|^2t}{2\pi}-(1-\varepsilon)\xi+\pi\sqrt{-1}
    \right)^2
  \right)
  \\
  &+\frac{u^2t}{2\pi}-\sqrt{-1}ut-(1-\varepsilon)u
\end{split}
\end{equation*}
Since the dilogarithm function $\Li(z)$ is analytic when $\Re(z)<1$, we have
\begin{equation*}
\begin{split}
  &\lim_{\varepsilon\searrow0}
  \Phi\left(\left(-\frac{u}{2\pi}+\sqrt{-1}\right)t+1-\varepsilon\right)
  \\
  =&
  \frac{1}{\xi}
  \left(
    -
    \Li_2(e^{-|\xi|^2t/(2\pi)})
    -
    \Li_2(e^{2u-|\xi|^2t/(2\pi)})
  \right)
  \\
  &-\frac{\pi^2}{6\xi}
  -\frac{1}{2\xi}
  \left(
    \frac{|\xi|^2t}{2\pi}-\pi\sqrt{-1}
  \right)^2
  \\
  &+\frac{u^2t}{2\pi}-\sqrt{-1}ut-u.
\end{split}
\end{equation*}
So we have
\begin{equation*}
\begin{split}
  &\lim_{\varepsilon\searrow0}
  \Re\left(\left(-\frac{u}{2\pi}+\sqrt{-1}\right)t+1-\varepsilon\right)
  \\
  =&
  -\frac{u}{|\xi|^2}
  \left(
    \Re\Li_2(e^{-|\xi|^2t/(2\pi)})
    +
    \Re\Li_2(e^{2u-|\xi|^2t/(2\pi)})
  \right)
  \\
  &-\frac{u\pi^2}{6|\xi|^2}
  -\frac{u}{2|\xi|^2}
   \frac{|\xi|^4t^2}{4\pi^2}
  +
  \frac{u\pi^2}{2|\xi|^2}
  -
  \frac{4\pi|\xi|^2t/(2\pi)\pi}{2|\xi|^2}
  +
  \frac{u^2t}{2\pi}
  -u
  \\
  =&
  -\frac{u}{|\xi|^2}
  \left(
    \Re\Li_2(e^{-|\xi|^2t/(2\pi)})
    +
    \Re\Li_2(e^{2u-|\xi|^2t/(2\pi)})
  \right)
  \\
  &+\frac{u\pi^2}{3|\xi|^2}
  -\frac{u|\xi|^2t^2}{8\pi^2}
  -\pi t
  +\frac{u^2t}{2\pi}
  -u
  \le0.
\end{split}
\end{equation*}
The last inequality follows since $\frac{u\pi^2}{3|\xi|^2}
  -\frac{u|\xi|^2t^2}{8\pi^2}
  -\pi t
  +\frac{u^2t}{2\pi}
  -u$,
is a quadratic function with respect to $u$ with non-positive maximum.
\item $2u-|\xi|^2t/(2\pi)>0$.
In this case we may choose $\varepsilon$ small so that $u-|\xi|^2t/(2\pi)+(1-\varepsilon)u>0$.
Then we have
\begin{equation*}
\begin{split}
  &\Phi\left(\left(-\frac{u}{2\pi}+\sqrt{-1}\right)t+1-\varepsilon\right)
  \\
  =&
  \frac{1}{\xi}
  \left(
    \vphantom{\left(\frac{|\xi|^2t}{2\pi}\right)^2}
    -
    \Li_2(e^{-u-|\xi|^2t/(2\pi)+(1-\varepsilon)\xi})
    +
    \Li_2(e^{-u+|\xi|^2t/(2\pi)-(1-\varepsilon)\xi})
  \right.
  \\
  &\left.
    -\frac{1}{2}
    \left(
      u+\frac{|\xi|^2t}{2\pi}-(1-\varepsilon)\xi+\pi\sqrt{-1}
    \right)^2
    +\frac{1}{2}
    \left(
      u-\frac{|\xi|^2t}{2\pi}+(1-\varepsilon)\xi-\pi\sqrt{-1}
    \right)^2
  \right)
  \\
  &+
  \frac{u^2t}{2\pi}-\sqrt{-1}ut-(1-\varepsilon)u
\end{split}
\end{equation*}
and so we have
\begin{equation*}
\begin{split}
  &\lim_{\varepsilon\searrow0}
  \Re\Phi\left(\left(-\frac{u}{2\pi}+\sqrt{-1}\right)t+1-\varepsilon\right)
  \\
  =&
  \frac{u}{|\xi|^2}
  \left(
    -
    \Li_2(e^{-|\xi|^2t/(2\pi)})
    +
    \Li_2(e^{-2u+|\xi|^2t/(2\pi)})
  \right)
  +\frac{2u^3}{|\xi|^2}
  -\frac{u^2t}{2\pi}
  +\frac{4\pi^2u}{|\xi|^2}
  -u.
\end{split}
\end{equation*}
Putting $s:=|\xi|^2t/(2\pi)$ and consider the function
\begin{equation*}
\begin{split}
  f(u,s)
  &:=
  -
  \Li_2(e^{-s})
  +
  \Li_2(e^{-2u+s})
  +2u^2
  -us
  +4\pi^2
  -|\xi|^2
  -\frac{2\pi^2s}{u}
  \\
  &=
  -
  \Li_2(e^{-s})
  +
  \Li_2(e^{-2u+s})
  +u^2
  -us-\frac{2\pi^2s}{u}
\end{split}
\end{equation*}
so that
\begin{equation*}
  \frac{u}{|\xi|^2}f\left(u,\frac{|\xi|^2t}{2\pi}\right)
  =
  \lim_{\varepsilon\searrow0}
  \Re\Phi\left(\left(-\frac{u}{2\pi}+\sqrt{-1}\right)t+1-\varepsilon\right)
  -\pi t.
\end{equation*}
\end{itemize}
\par
We will show $f(u,s)\le0$ for $0\le s<2u\le2$.
\par
We have
\begin{equation*}
  \exp\left[-\frac{\partial\,f(u,t)}{\partial\,t}\right]
  =
  2e^{2\pi^2/u}
  \left(
    \cosh(u)-\cosh(s-u)
  \right).
\end{equation*}
Therefore it can be shown that for fixed $u$, $f(u,s)$ is increasing for $0\le s<u-\arccosh(\cosh(u)-\exp(-2\pi^2/u)/2)$, decreasing for $u-\arccosh(\cosh(u)-\exp(-2\pi^2/u)/2)<s<u+\arccosh(\cosh(u)-\exp(-2\pi^2/u)/2)$, and increasing for $u+\arccosh(\cosh(u)-\exp(-2\pi^2/u)/2)<s<u$.
Since a graph of $f\bigl(u,u-\arccosh(\cosh(u)-\exp(-2\pi^2/u)/2)\bigr)$ looks as Figure~\ref{fig:graph} and
\begin{equation*}
  f(u,2u)
  =
  -\frac{23\pi^2}{6}
  -u^2
  -\Li_2(e^{-2u})
  <0,
\end{equation*}
we see $f(u,s)\le0$.
\begin{figure}[h]
  \includegraphics[scale=0.7]{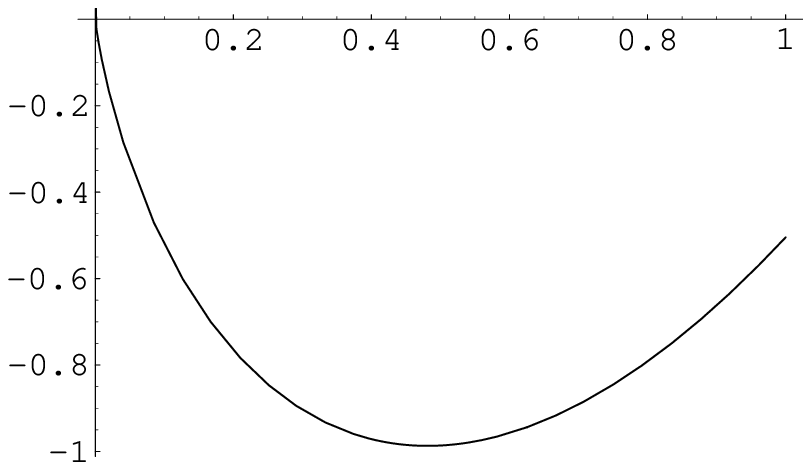}
  \caption{A graph of $f\bigl(u,u-\arccosh(\cosh(u)-\exp(-4\pi^2/u)/2)\bigr)$.}
  \label{fig:graph}
\end{figure}
\par
Therefore we finally have
\begin{equation*}
\begin{split}
  |I_{+,3}(N)|
  &<
  \frac{2}{1-e^{-\pi^2/u}}
  \exp\left(2A+B(1+e^{(u+|\xi|^2/(2\pi))R})|\gamma|\right)
  \int_{4\pi u/|\xi|^2}^{1}
  e^{-2N\pi t}\,dt
  \\
  &=
  \frac{e^{-8N\pi^2u/|\xi|^2}-e^{-2N\pi}}{N\pi(1-e^{-\pi^2/u})}
  \exp\left(2A+B(1+e^{(u+|\xi|^2/(2\pi))R})|\gamma|\right)
  \\
  &<
  \frac{K_{+,3}}{N}
\end{split}
\end{equation*}
for a positive constant $K_{+,3}$.
\end{proof}%{of I_{+,3}}
%%%%%%%%%%%%%%%%%%%%%%%%%%%%%%%%%%%%%%%%%%%%%%%%%%%%%%%%%%%%%%%%%%%
\begin{proof}[Proof of \eqref{eq:I_{-,3}}]
Since $\tan$ has period $\pi$ we have
\begin{equation}\label{eq:tan_I-3}
  |\tan(N\pi((u/(2\pi)-\sqrt{-1})t+1)+\sqrt{-1}|
  <
  \frac{2e^{-2N\pi t}}{1-e^{-\pi^2/u}}
\end{equation}
from \eqref{eq:I_{+,1}tan}
\par
From \eqref{eq:tan_I-3} we have
\begin{equation}\label{eq:I_{-,3}integral}
  \left|I_{-,3}(N)\right|
  \le
  \frac{2}{1-e^{-\pi^2/u}}
  \int_{0}^{1}
  e^{-2N\pi t}
  \left|
    g_N\left(\left(-\frac{u}{2\pi}+\sqrt{-1}\right)t+1\right)
  \right|
  \,dt.
\end{equation}
\par
As in the case of $I_{+,1}(N)$ we consider the integral on the segment $(u/(2\pi)-\sqrt{-1})t+1-\varepsilon$ ($0\le t\le1$) for small $\varepsilon$.
\par
We estimate $\Phi\bigl((u/(2\pi)-\sqrt{-1})t+1-\varepsilon\bigr)$.
We have
\begin{equation*}
\begin{split}
  &\Phi\left(\left(\frac{u}{2\pi}-\sqrt{-1}\right)t+1-\varepsilon\right)
  \\
  =&
  \frac{1}{\xi}
  \left(
    \Li_2(e^{u-|\xi|^2t/(2\pi)-(1-\varepsilon)\xi})
    -
    \Li_2(e^{u+|\xi|^2t/(2\pi)+(1-\varepsilon)\xi})
  \right)
  \\
  &-\frac{u^2t}{2\pi}+\sqrt{-1}ut-(1-\varepsilon)u.
\end{split}
\end{equation*}
We will calculate $\lim_{\varepsilon\searrow0}\Phi\bigl((u/(2\pi)-\sqrt{-1})t+1-\varepsilon\bigr)$.
\par
Note that $\Re(u+|\xi|^2t/(2\pi)+(1-\varepsilon)\xi)=2u+|\xi|^2t/(2\pi)-\varepsilon u>0$ for small $\varepsilon$.
Since $\Re(u-|\xi|^2t/(2\pi)-(1-\varepsilon)\xi)=-|\xi|^2t/(2\pi)+\varepsilon u$, if $t>0$ we assume that $\Re(u-|\xi|^2t/(2\pi)-(1-\varepsilon)\xi)<0$.
Therefore we assume that $t>0$.
\par
In this case from \eqref{eq:dilog} we have
\begin{equation*}
\begin{split}
  &\Phi\left(\left(\frac{u}{2\pi}-\sqrt{-1}\right)t+1-\varepsilon\right)
  \\
  =&
  \frac{1}{\xi}
  \left(
    \Li_2(e^{-|\xi|^2t/(2\pi)+\varepsilon\xi})
    +
    \Li_2(e^{-2u-|\xi|^2t/(2\pi)+\varepsilon\xi})
  \right)
  \\
  &
  +\frac{\pi^2}{6\xi}
  +\frac{1}{2\xi}
  \left(
    2u+\frac{|\xi|^2t}{2\pi}-\varepsilon\xi+\pi\sqrt{-1}
  \right)^2
  \\
  &-\frac{u^2t}{2\pi}+\sqrt{-1}ut-(1-\varepsilon)u
\end{split}
\end{equation*}
and
\begin{equation*}
\begin{split}
  &\Re\Phi\left(\left(\frac{u}{2\pi}-\sqrt{-1}\right)t+1-\varepsilon\right)
  \\
  =&
  \frac{1}{\xi}
  \left(
    \Li_2(e^{-|\xi|^2t/(2\pi)+\varepsilon\xi})
    +
    \Li_2(e^{-2u-|\xi|^2t/(2\pi)+\varepsilon\xi})
  \right)
  \\
  &
  +\frac{\pi^2}{6\xi}
  +\frac{1}{2\xi}
  \left(
    2u+\frac{|\xi|^2t}{2\pi}-\varepsilon\xi+\pi\sqrt{-1}
  \right)^2
  \\
  &-\frac{u^2t}{2\pi}+\sqrt{-1}ut-(1-\varepsilon)u.
\end{split}
\end{equation*}
So we have
\begin{equation*}
\begin{split}
  &\lim_{\varepsilon\searrow0}
  \Phi\left(\left(\frac{u}{2\pi}-\sqrt{-1}\right)t+1-\varepsilon\right)
  \\
  =&
  \frac{1}{\xi}
  \left(
    \Li_2(e^{-|\xi|^2t/(2\pi)})
    +
    \Li_2(e^{-2u-|\xi|^2t/(2\pi)})
  \right)
  \\
  &
  +\frac{\pi^2}{6\xi}
  +\frac{1}{2\xi}
  \left(
    2u+\frac{|\xi|^2t}{2\pi}+\pi\sqrt{-1}
  \right)^2
  \\
  &-\frac{u^2t}{2\pi}+\sqrt{-1}ut-u
  \\
  =&
  \frac{u}{|\xi|^2}
  \left(
    \Li_2(e^{-|\xi|^2t/(2\pi)})
    +
    \Li_2(e^{-2u-|\xi|^2t/(2\pi)})
  \right)
  \\
  &+\frac{u\pi^2}{6|\xi|^2}
  +
  \frac{u}{2|\xi|^2}
  \left(
    2u+\frac{|\xi|^2t}{2\pi}
  \right)^2
  -
  \frac{u\pi^2}{2|\xi|^2}
  +
  \frac{4\pi^2}{2|\xi|^2}
  \left(
    2u+\frac{|\xi|^2t}{2\pi}
  \right)
  -
  \frac{u^2t}{2\pi}
  -u
  \\
  \le
  &\frac{11u\pi^2}{3|\xi|^2}
  +\frac{2u^3}{|\xi|^2}
  +\frac{u|\xi|^2t^2}{8\pi^2}
  +\pi t
  +\frac{u^2t}{2\pi}
  -u
  <\frac{3}{2}\pi t
\end{split}
\end{equation*}
if $0<u<1$ and $0<t\le1$.
\par
Therefore we finally have
\begin{equation*}
\begin{split}
  |I_{-,3}(N)|
  &<
  \frac{2}{1-e^{-\pi^2/u}}
  \exp\left(2A+B(1+e^{(u+|\xi|^2/(2\pi))R})|\gamma|\right)
  \int_{0}^{1}
  e^{-N\pi t/2}
  \,dt
  \\
  &=
  \frac{4(1-e^{-N\pi/2})}{N\pi(1-e^{-\pi^2/u)}}
  \exp\left(2A+B(1+e^{(u+|\xi|^2/(2\pi))R})|\gamma|\right)
  \\
  &<
  \frac{K_{-,3}}{N}
\end{split}
\end{equation*}
for a positive constant $K_{-,3}$.
\end{proof}%{of I_{-,3}}
%%%%%%%%%%%%%%%%%%%%%%%%%%%%%%%%%%%%%%%%%%%%%%%%%%%%%%%%%%%
\begin{proof}[Proof of \eqref{eq:I_{+,2}}]
From \cite[Equation~(4.6)]{Andersen/Hansen:JKNOT2006} we have
\begin{equation*}
  |I_{+,2}(N)|
  \le
  4e^{-2\pi N}
  \int_{\varepsilon}^{1-\varepsilon}|g_N(-u/(2\pi)+\sqrt{-1}+t)|\,dt.
\end{equation*}
From \eqref{eq:g_N} we have
\begin{multline*}
  |g_N(-u/(2\pi)+\sqrt{-1}+t)|
  \\
  \le
  \exp\bigl(N\Re(\Phi(-u/(2\pi)+\sqrt{-1}+t))\bigr)
  \exp\bigl(4A+2|\gamma|(1+e^{R(u+u^2/(2\pi)+2\pi)})\bigr).
\end{multline*}
\par
Now estimate $\Re(\Phi(-u/(2\pi)+\sqrt{-1}+t))$.
We have
\begin{equation*}
\begin{split}
  &\Phi\left(-\frac{u}{2\pi}+\sqrt{-1}+t\right)
  \\
  =&
  \frac{1}{\xi}
  \left(
    \Li_2(e^{u+|\xi|^2/(2\pi)-\xi t})
    -
    \Li_2(e^{u-|\xi|^2/(2\pi)+\xi t})
  \right)
  +\frac{u^2}{2\pi}-\sqrt{-1}u-ut.
\end{split}
\end{equation*}
Since $\Re(u+|\xi|^2/(2\pi)-\xi t)=u(1-t)+|\xi|^2/(2\pi)>0$ and $\Re(u-|\xi|^2/(2\pi)+\xi t)=u(1+t)-|\xi|^2/(2\pi)\le0$, we have
\begin{equation*}
\begin{split}
  &\Phi\left(-\frac{u}{2\pi}+\sqrt{-1}+t\right)
  \\
  =&
  -\frac{1}{\xi}
  \left(
    \Li_2(e^{u+|\xi|^2/(2\pi)-\xi t})
    +
    \Li_2(e^{u-|\xi|^2/(2\pi)+\xi t})
  \right)
  \\
  &
  -\frac{\pi^2}{6\xi}
  -\frac{1}{2\xi}
  \left(
    u+\frac{|\xi|^2}{2\pi}-(u+2\pi\sqrt{-1})t+\pi\sqrt{-1}
  \right)^2
  +\frac{u^2}{2\pi}-\sqrt{-1}u-ut
\end{split}
\end{equation*}
from \eqref{eq:dilog}.
Therefore we have
\begin{equation*}
\begin{split}
  &\Re\Phi\left(-\frac{u}{2\pi}+\sqrt{-1}+t\right)
  \\
  =&
  -\frac{u}{|\xi|^2}
  \left(
    \Re\Li_2(e^{u+|\xi|^2/(2\pi)-\xi t})
    +
    \Re\Li_2(e^{u-|\xi|^2/(2\pi)+\xi t})
  \right)
  \\
  &-\frac{2\pi}{|\xi|^2}
  \left(
    \Im\Li_2(e^{u+|\xi|^2/(2\pi)-\xi t})
    +
    \Im\Li_2(e^{u-|\xi|^2/(2\pi)+\xi t})
  \right)
  \\
  &
  +(2t-1)\pi
  -\frac{1}{2}(t^2+1)u
  +\frac{u^2t}{2\pi}
  -\frac{u^3}{8\pi^2}
  +\frac{\pi^2u}{3|\xi|^2}
  \\
  \le&
  -\frac{u}{|\xi|^2}
  \left(
    \Re\Li_2(e^{u+|\xi|^2/(2\pi)-\xi t})
    +
    \Re\Li_2(e^{u-|\xi|^2/(2\pi)+\xi t})
  \right)
  \\
  &-\frac{2\pi}{|\xi|^2}
  \left(
    \Im\Li_2(e^{u+|\xi|^2/(2\pi)-\xi t})
    +
    \Im\Li_2(e^{u-|\xi|^2/(2\pi)+\xi t})
  \right)
  \\
  &+(2t-1)\pi.
\end{split}
\end{equation*}
For $0<r<1$ and $0<\theta<2\pi$ we have
\begin{equation*}
\begin{split}
  \Re\Li_2(re^{\sqrt{-1}\theta})
  &=
  -\frac{1}{2}
  \int_{0}^{r}
  \frac{\log(1-2s\cos\theta+s^2)}{s}\,ds
  \ge
  -\frac{1}{2}
  \int_{0}^{r}
  \frac{\log(1+2s+s^2)}{s}\,ds
  \\
  &=
  -\int_{0}^{r}\frac{\log(1+s)}{s}\,ds
  \ge
  -\int_{0}^{1}\frac{\log(1+s)}{s}\,ds
  =
  \Li_2(-1)
  =
  -\frac{\pi^2}{12}.
\end{split}
\end{equation*}
We also have
\begin{equation*}
  \Im\Li_2(re^{\sqrt{-1}\theta})
  =
  -\int_{0}^{r}
  \frac{\arg(1-se^{\sqrt{-1}\theta})}{s}\,ds.
\end{equation*}
If $0<\theta\le\pi$ the right hand side is non-negative.
If $\pi<\theta<2\pi$ we have
\begin{equation*}
\begin{split}
  \Im\Li_2(re^{\sqrt{-1}\theta})
  &=
  -\int_{0}^{1}\frac{\arg(1-se^{\sqrt{-1}\theta})}{s}\,ds
  +
  \int_{r}^{1}
  \frac{\arg(1-se^{\sqrt{-1}\theta})}{s}\,ds
  \\
  &=
  \Im\Li_2(e^{\sqrt{-1}\theta})
  +
  \int_{r}^{1}
  \frac{\arg(1-se^{\sqrt{-1}\theta})}{s}\,ds.
\end{split}
\end{equation*}
The second integral is positive and $\Im\Li_2(e^{\sqrt{-1}\theta})$ is bigger than or equal to $-\Im\Li_2(\exp(\sqrt{-1}\pi/3))=-1.01494\ldots$.
\par
Therefore we have
\begin{equation*}
  \Re\Phi\left(-\frac{u}{2\pi}+\sqrt{-1}+t\right)
  \le
  \frac{\pi^2u}{6|\xi|^2}
  +
  \frac{2\pi\Im\Li_2(e^{\sqrt{-1}\pi/3})}{|\xi|^2}
  +(2t-1)\pi
\end{equation*}
So we have
\begin{equation*}
\begin{split}
  &|I_{+,2}(N)|
  \\
  \le&
  4\exp(4A+2|\gamma|(1+e^{R(u+u^2/(2\pi)+2\pi)}))
  e^{-3\pi N}
  \\
  &\times
  \exp
  \left[
  N
  \left(
    \frac{\pi^2u}{6|\xi|^2}
    +
    \frac{2\pi\Im\Li_2(e^{\sqrt{-1}\pi/3})}{|\xi|^2}
  \right)
  \right]
  \int_{0}^{1}e^{2N\pi t}\,dt
  \\
  =&
  \frac{2(1-e^{-2\pi N})}{\pi N}
  \exp(4A+2|\gamma|(1+e^{R(u+u^2/(2\pi)+2\pi)}))
  \\
  &\times
  \exp
  \left[
  \pi N
  \left(
    \frac{\pi u}{6|\xi|^2}
    +
    \frac{2\Im\Li_2(e^{\sqrt{-1}\pi/3})}{|\xi|^2}
    -1
  \right)
  \right]
  \\
  <&
  \frac{2}{\pi N}
  \exp(4A+2|\gamma|(1+e^{R(u+u^2/(2\pi)+2\pi)}))
  \\
  <&
  \frac{K_{+,2}}{N}
\end{split}
\end{equation*}
for a positive constant $K_{+,2}$.
Here we use the inequality
\begin{equation*}
  \frac{\pi u}{6|\xi|^2}
  +
  \frac{2\Im\Li_2(e^{\sqrt{-1}\pi/3})}{|\xi|^2}
  -1
  <
  \frac{\pi}{6}
  +\frac{1.01494\ldots}{2\pi^2}
  -1
  <0.
\end{equation*}
\end{proof}%{of I_{+,2}}
%%%%%%%%%%%%%%%%%%%%%%%%%%%%%%%%%%%%%%%%%%%%%%%%%%%%%%%%%%%
\begin{proof}[Proof of \eqref{eq:I_{-,2}}]
From \cite[Equation~(4.6)]{Andersen/Hansen:JKNOT2006} we have
\begin{equation*}
  |I_{-,2}(N)|
  \le
  4e^{-2\pi N}
  \int_{\varepsilon}^{1-\varepsilon}|g_N(u/(2\pi)-\sqrt{-1}+t)|\,dt.
\end{equation*}
From \eqref{eq:g_N} we have
\begin{multline*}
  |g_N(u/(2\pi)-\sqrt{-1}+t)|
  \\
  \le
  \exp\bigl(N\Re(\Phi(u/(2\pi)-\sqrt{-1}+t))\bigr)
  \exp\bigl(4A+2|\gamma|(1+e^{R(u+u^2/(2\pi)+2\pi)})\bigr).
\end{multline*}
\par
Now estimate $\Re(\Phi(u/(2\pi)-\sqrt{-1}+t))$.
We have
\begin{equation*}
\begin{split}
  &\Phi\left(\frac{u}{2\pi}-\sqrt{-1}+t\right)
  \\
  =&
  \frac{1}{\xi}
  \left(
    \Li_2(e^{u-|\xi|^2/(2\pi)-\xi t})
    -
    \Li_2(e^{u+|\xi|^2/(2\pi)+\xi t})
  \right)
  -\frac{u^2}{2\pi}+\sqrt{-1}u-ut.
\end{split}
\end{equation*}
Since $\Re(u-|\xi|^2/(2\pi)-\xi t)=u(1-t)-|\xi|^2/(2\pi)<0$ and $\Re(u+|\xi|^2/(2\pi)+\xi t)=u(1+t)+|\xi|^2/(2\pi)>0$, we have
\begin{equation*}
\begin{split}
  &\Phi\left(\frac{u}{2\pi}-\sqrt{-1}+t\right)
  \\
  =&
  \frac{1}{\xi}
  \left(
    \Li_2(e^{u-|\xi|^2/(2\pi)-\xi t})
    +
    \Li_2(e^{-u-|\xi|^2/(2\pi)-\xi t})
  \right)
  \\
  &
  +\frac{\pi^2}{6\xi}
  +\frac{1}{2\xi}
  \left(
    u+\frac{|\xi|^2}{2\pi}+(u+2\pi\sqrt{-1})t-\pi\sqrt{-1}
  \right)^2
  -\frac{u^2}{2\pi}+\sqrt{-1}u-ut
\end{split}
\end{equation*}
from \eqref{eq:dilog}.
Therefore we have
\begin{equation*}
\begin{split}
  &\Re\Phi\left(\frac{u}{2\pi}-\sqrt{-1}+t\right)
  \\
  =&
  \frac{u}{|\xi|^2}
  \left(
    \Re\Li_2(e^{u-|\xi|^2/(2\pi)-\xi t})
    +
    \Re\Li_2(e^{-u-|\xi|^2/(2\pi)-\xi t})
  \right)
  \\
  &
  +
  \frac{2\pi}{|\xi|^2}
  \left(
    \Im\Li_2(e^{u-|\xi|^2/(2\pi)-\xi t})
    +
    \Im\Li_2(e^{-u-|\xi|^2/(2\pi)-\xi t})
  \right)
  \\
  &
  +\frac{u\pi^2}{6|\xi|^2}-\frac{u^2}{2\pi}-ut
  \\
  &
  +\frac{u}{2|\xi|^2}
  \left(
    \left(
    u+\frac{|\xi|^2}{2\pi}+ut
    \right)^2
    -
    \pi^2(2t-1)^2
  \right)
  \\
  &
  +
  \frac{4\pi}{2|\xi|^2}
  \pi(2t-1)
  \left(
    u+\frac{|\xi|^2}{2\pi}+ut
  \right)
  \\
  =&
  \frac{u}{|\xi|^2}
  \left(
    \Re\Li_2(e^{u-|\xi|^2/(2\pi)-\xi t})
    +
    \Re\Li_2(e^{-u-|\xi|^2/(2\pi)-\xi t})
  \right)
  \\
  &
  +
  \frac{2\pi}{|\xi|^2}
  \left(
    \Im\Li_2(e^{u-|\xi|^2/(2\pi)-\xi t})
    +
    \Im\Li_2(e^{-u-|\xi|^2/(2\pi)-\xi t})
  \right)
  \\
  &
  +\frac{ut^2}{2}
  +\frac{|\xi|^2t}{2\pi}
  -\frac{7u\pi^2}{3|\xi|^2}
  +\frac{u^3}{2|\xi|^2}
  +\frac{u|\xi|^2}{8\pi^2}
  -\pi.
\end{split}
\end{equation*}
For $0<r<1$ and $0<\theta<2\pi$ we have
\begin{equation*}
\begin{split}
  \Re\Li_2(re^{\sqrt{-1}\theta})
  &=
  -\frac{1}{2}
  \int_{0}^{r}
  \frac{\log(1-2s\cos\theta+s^2)}{s}\,ds
  \le
  -\frac{1}{2}
  \int_{0}^{r}
  \frac{\log(1-2s+s^2)}{s}\,ds
  \\
  &=
  -\int_{0}^{r}\frac{\log(1-s)}{s}\,ds
  \le
  -\int_{0}^{1}\frac{\log(1-s)}{s}\,ds
  =
  \Li_2(1)
  =
  \frac{\pi^2}{6}.
\end{split}
\end{equation*}
We also have
\begin{equation*}
\begin{split}
  \Im\Li_2(re^{\sqrt{-1}\theta})
  &=
  -\int_{0}^{r}
  \frac{\arg(1-se^{\sqrt{-1}\theta})}{s}\,ds
  \\
  &=
  -\int_{0}^{1}\frac{\arg(1-se^{\sqrt{-1}\theta})}{s}\,ds
  +
  \int_{r}^{1}
  \frac{\arg(1-se^{\sqrt{-1}\theta})}{s}\,ds
  \\
  &\le
  \Im\Li_2(e^{\sqrt{-1}\theta})
  \le
  \Im\Li_2(e^{\sqrt{-1}\pi/3})
  =
  1.01494\dots.
\end{split}
\end{equation*}
when $0\le\theta\le\pi$.
Since $\Im\Li_2(re^{\sqrt{-1}\theta})=-\Im\Li_2(re^{\sqrt{-1}(\theta-\pi)})$ when $\pi<\theta\le2\pi$, we have
\begin{equation*}
\begin{split}
  &\Re\Phi\left(\frac{u}{2\pi}-\sqrt{-1}+t\right)
  \\
  =&
  \frac{4\pi}{|\xi|^2}\Im\Li_2(e^{\sqrt{-1}\pi/3})
  +\frac{ut^2}{2}
  +\frac{|\xi|^2t}{2\pi}
  -2\frac{u\pi^2}{|\xi|^2}
  +\frac{u^3}{2|\xi|^2}
  +\frac{u|\xi|^2}{8\pi^2}
  -\pi
  \\
  &\quad\text{(since $0<t<1$ and $0<u<\pi/2$)}.
  \\
  <&
  \frac{|\xi|^2t}{2\pi}
  +\frac{1}{\pi}\Im\Li_2(e^{\sqrt{-1}\pi/3})
  +\frac{\pi}{2}
  -\frac{4\pi}{17}
  +\frac{\pi}{64}
  +\frac{17\pi}{64}
  -\pi
  \\
  \le&
  \frac{|\xi|^2t}{2\pi}
  +
  \frac{1}{\pi}\Im\Li_2(e^{\sqrt{-1}\pi/3})
  -\frac{247}{544}\pi.
\end{split}
\end{equation*}
when $u<\pi/2$.
\par
So we have
\begin{equation*}
\begin{split}
  &|I_{-,2}(N)|
  \\
  <&
  4e^{(\Im\Li_2(e^{\sqrt{-1}\pi/3})/\pi-247\pi/544-2\pi)N}
  \exp\bigl(4A+2|\gamma|(1+e^{R(u+u^2/(2\pi)+2\pi)})\bigr)
  \\
  &\times
  \lim_{\varepsilon\searrow0}
  \int_{\varepsilon}^{1-\varepsilon}e^{N|\xi|^2t/(2\pi)}\,dt
  \\
  =&
  4e^{(\Im\Li_2(e^{\sqrt{-1}\pi/3})/\pi-247\pi/544-2\pi)N}
  \exp\bigl(4A+2|\gamma|(1+e^{R(u+u^2/(2\pi)+2\pi)})\bigr)
  \\
  &\times
  \frac{2\pi}{N|\xi|^2}(e^{N|\xi|^2/(2\pi)}-1)
  \\
  =&
  \frac{8\pi\exp\bigl(4A+2|\gamma|(1+e^{R(u+u^2/(2\pi)+2\pi)})\bigr)}{N|\xi|^2}
  \\
  &\times
  \left(
    e^{(\Im\Li_2(e^{\sqrt{-1}\pi/3})/\pi-247\pi/544-2\pi+|\xi|^2/(2\pi))N}
  \right.
  \\
  &\phantom{\times}\qquad
  \left.
    -
    e^{(\Im\Li_2(e^{\sqrt{-1}\pi/3})/\pi-247\pi/544-2\pi)N}
  \right)
  \\
  <&
  \frac{8\pi\exp\bigl(4A+2|\gamma|(1+e^{R(u+u^2/(2\pi)+2\pi)})\bigr)}{N|\xi|^2}
  \\
  &\times
  \left(
    e^{(\Im\Li_2(e^{\sqrt{-1}\pi/3})/\pi-179\pi/544)N}
    -
    e^{(\Im\Li_2(e^{\sqrt{-1}\pi/3})/\pi-247\pi/544-2\pi)N}
  \right)
  \\
  <&
  \frac{8\pi\exp\bigl(4A+2|\gamma|(1+e^{R(u+u^2/(2\pi)+2\pi)})\bigr)}{N|\xi|^2}
  \\
  <&
  \frac{K_{-,2}}{N}
\end{split}
\end{equation*}
for a positive constant $K_{-,2}$, since $\Im\Li_2(e^{\sqrt{-1}\pi/3})/\pi-179\pi/544=-0.710657\dots$.
\end{proof}%{of I_{-,2}}
%%%%%%%%%%%%%%%%%%%%%%%%%%%%%%%%%%%%%%%%%%%%%%%%%%%%%%%%%%%%%
\section{Proof of Proposition~\ref{prop:4.9}}\label{sec:prop:4.9}
In this section we again follow \cite{{Andersen/Hansen:JKNOT2006}} to prove Proposition~\ref{prop:4.9}.
From \eqref{eq:g_N} we have
\begin{equation*}
  \int_{p(\varepsilon)}g_N(w)\,dw
  =
  \int_{p(\varepsilon)}\exp\bigl(N\Phi(w)\bigr)\,dw
  +
  \int_{p(\varepsilon)}\exp\bigl(N\Phi(w)\bigr)h_{\gamma}(w)\,dw
\end{equation*}
for any path in the parallelogram bounded by $C_+(\varepsilon)\cup C_-(\varepsilon)$ connecting $\varepsilon$ and $1-\varepsilon$, where
\begin{equation*}
  h_{\gamma}(w)
  :=
  \exp
  \left(
    I_{\gamma}(\pi-\sqrt{-1}u+\sqrt{-1}\xi w)
    -
    I_{\gamma}(-\pi-\sqrt{-1}u-\sqrt{-1}\xi w)
  \right)
  -1.
\end{equation*}
Note that $h_{\gamma}(w)$ is defined for $w$ with $0<\Im(\xi w)<2\pi$, that is, for $w$ between the two parallel thick lines depicted in Figure~\ref{fig:contour}.
\par
Then we have
\begin{equation}\label{eq:max}
\begin{split}
  &\left|
    \int_{p(\varepsilon)}g_N(w)\,dw
    -
    \int_{p(\varepsilon)}\exp\bigl(N\Phi(w)\bigr)\,dw
  \right|
  \\
  \le&
  \int_{p(\varepsilon)}\left|\exp\big(N\Phi(w)\bigr)h_{\gamma}(w)\right|\,dw
  \\
  \le&
  \max_{w\in p(\varepsilon)}\left\{\exp\bigl(N\Re\Phi(w)\bigr)\right\}
  \int_{p(\varepsilon)}|h_{\gamma}(w)|\,dw
  \\
  =&
  \max_{w\in p(\varepsilon)}\left\{\exp\bigl(N\Re\Phi(w)\bigr)\right\}
  \int_{\varepsilon}^{1-\varepsilon}|h_{\gamma}(w)|\,dw
  \\
  \le&
  \max_{w\in p(\varepsilon)}\left\{\exp\bigl(N\Re\Phi(w)\bigr)\right\}
  \int_{0}^{1}|h_{\gamma}(w)|\,dw.
\end{split}
\end{equation}
Here we use the analyticity of $h_{\gamma}$ in the equality.
\par
By the definition of $h_{\gamma}$ we have
\begin{equation}\label{eq:h_gamma}
  h_{\gamma}(t)
  =
  \sum_{n=1}^{\infty}
  \frac{1}{n!}
  \left(
    I_{\gamma}(\pi-\sqrt{-1}u+\sqrt{-1}\xi t)
    -
    I_{\gamma}(-\pi-\sqrt{-1}u-\sqrt{-1}\xi t)
  \right)^n.
\end{equation}
From Lemma~\ref{lem:Lemma3}, for $0<t<1$ we have
\begin{equation*}
  \left|
    I_{\gamma}(\pi-\sqrt{-1}u+\sqrt{-1}\xi t)
  \right|
  \le
  A|\gamma|
  \left(
    \frac{1}{2\pi t}+\frac{1}{2\pi-2\pi t}
  \right)
  +
  B|\gamma|
  \left(
    1+e^{(u-ut)R}
  \right)
\end{equation*}
and
\begin{equation*}
  \left|
    I_{\gamma}(-\pi-\sqrt{-1}u-\sqrt{-1}\xi w)
  \right|
  \le
  A|\gamma|
  \left(
    \frac{1}{2\pi-2\pi t}+\frac{1}{2\pi t}
  \right)
  +
  B|\gamma|
  \left(
    1+e^{(u+ut)R}
  \right),
\end{equation*}
and so we have
\begin{equation*}
\begin{split}
  &\left|
    I_{\gamma}(\pi-\sqrt{-1}u+2\pi t)
    -
    I_{\gamma}(-\pi-\sqrt{-1}u-2\pi t)
  \right|
  \\
  \le&
  \frac{A|\gamma|}{\pi}
  \left(
    \frac{1}{1-t}+\frac{1}{t}
  \right)
  +
  B|\gamma|
  \left(
    2+e^{(1+t)uR}+e^{(1-t)uR}
  \right)
  \\
  \le&
  |\gamma|\left(A'f(t)+B'\right)
\end{split}
\end{equation*}
for some positive constants $A'$ and $B'$, where we put $f(t):=1/t+1/(1-t)$.
Since $f(t)\ge4$ for $0<t<1$ we have
\begin{equation}\label{eq:I_gamma}
\begin{split}
  &\left|
    I_{\gamma}(\pi-\sqrt{-1}u+\sqrt{-1}\xi w)
    -
    I_{\gamma}(-\pi-\sqrt{-1}u-\sqrt{-1}\xi w)
  \right|
  \\
  \le&
  |\gamma|\left(A'f(t)+B'\frac{f(t)}{4}\right)
  \\
  =&
  A''|\gamma|f(t),
\end{split}
\end{equation}
where $A'':=A'+B'/4$.
\par
From the argument in \cite[P.~537]{Andersen/Hansen:JKNOT2006} we have
\begin{equation*}
  \int_{|\gamma|}^{1-|\gamma|}f(t)^n\,dt
  \le
  2^{2n+1}\int_{|\gamma|}^{1/2}\frac{dt}{t^n}
\end{equation*}
for $n\ge1$.
Since $|\gamma|=|\xi|/(2N)$ we have
\begin{equation*}
  \int_{|\gamma|}^{1/2}\frac{dt}{t}
  =
  \log{N}-\log|\xi|
  \le
  \log{N}
\end{equation*}
and
\begin{equation*}
  \int_{|\gamma|}^{1/2}\frac{dt}{t^n}
  =
  \frac{1}{n-1}
  \left(
    \frac{1}{|\gamma|^{n-1}}-2^{n-1}
  \right)
  \le
  \frac{1}{(n-1)|\gamma|^{n-1}}.
\end{equation*}
for $n\ge2$.
\par
Therefore from \eqref{eq:h_gamma} and \eqref{eq:I_gamma} we have
\begin{equation*}
\begin{split}
  \int_{|\gamma|}^{1-|\gamma|}|h_{\gamma}(t)|\,dt
  &\le
  \sum_{n=1}^{\infty}
  \frac{1}{n!}(A'')^n|\gamma|^n
  \int_{|\gamma|}^{1-|\gamma|}f(t)^n\,dt
  \\
  &\le
  2|\gamma|
  \left(
    4A''\log{N}
    +
    \sum_{n=2}^{\infty}
    \frac{(4A'')^n}{(n-1)n!}
  \right)
  \\
  &\le
  \frac{|\xi|}{N}
  \left(
    4A''\log{N}
    +
    \exp(4A'')-4A''-1
  \right)
  \\
  &\le
  \frac{K'\log{N}}{N}
\end{split}
\end{equation*}
for a positive constant $K'$ if $N$ is sufficiently large.
\par
For $0\le t\le1$, we also have from Lemma~\ref{lem:Lemma3}
\begin{equation*}
\begin{split}
  &\left|
    I_{\gamma}(\pi-\sqrt{-1}u+2\pi t)
    -
    I_{\gamma}(-\pi-\sqrt{-1}u-2\pi t)
  \right|
  \\
  \le&
  4A
  +
  B|\gamma|
  \left(
    2+e^{(1+t)uR}+e^{(1-t)uR}
  \right).
\end{split}
\end{equation*}
Since $|\gamma|=|\xi|/(2N)$, we have
\begin{equation*}
  |h_{\gamma}(t)|
  \le
  \exp
  \left(
    4A
    +
    \frac{B'|\xi|}{2N}
  \right).
\end{equation*}
So we have
\begin{equation*}
  \int_{0}^{|\gamma|}|h_{\gamma}(t)|\,dt
  \le
  \frac{|\xi|}{2N}
  \exp
  \left(
    4A
    +
    \frac{B'|\xi|}{2N}
  \right)
  \le
  \frac{K''}{N}
\end{equation*}
and
\begin{equation*}
  \int_{1-|\gamma|}^{1}|h_{\gamma}(t)|\,dt
  \le
  \frac{|\xi|}{2N}
  \exp
  \left(
    4A
    +
    \frac{B'|\xi|}{2N}
  \right)
  \le
  \frac{K'''}{N}
\end{equation*}
for positive constants $K''$ and $K'''$.
\par
Therefore we have
\begin{equation*}
  \int_{0}^{1}|h_{\gamma}(t)|\,dt
  \le
  \frac{K'\log{N}}{N}+\frac{K''}{N}+\frac{K'''}{N}
  \le
  \frac{K_2\log{N}}{N}
\end{equation*}
for a positive constant $K_2$.
Now from \eqref{eq:max} we finally have
\begin{equation*}
  \left|
    \int_{p(\varepsilon)}g_N(w)\,dw
    -
    \int_{p(\varepsilon)}\exp\bigl(N\Phi(w)\bigr)\,dw
  \right|
  \le
  \frac{K_2\log{N}}{N}
  \max_{w\in p(\varepsilon)}\left\{\exp\bigl(N\Re\Phi(w)\bigr)\right\},
\end{equation*}
proving Proposition~\ref{prop:4.9}.
%%%%%%%%%%%%%%%%%%%%%%%%%%%%%%%%%%%%%%%%%%%%%%%%%%%%%%%%%%%%%%%%%
\section{Proof of Lemma~\ref{lem:Phi_0}}
\label{sec:Phi_0}
We will show that $\xi\Phi(w_0)$ is purely imaginary with positive imaginary part.
Then since $\xi$ is in the first quadrant, $\Phi(w_0)$ is in also in the first quadrant and so $\Re\Phi(w_0)>0$.
\par
Since $\varphi(u)$ is purely imaginary (Remark~\ref{rem:varphi}), we have $\Li_2(e^{u-\varphi(u)})=\overline{\Li_2(e^{u+\varphi(u)})}$.
Therefore from \eqref{eq:Phi_0} and \eqref{eq:dilog} we have
\begin{equation*}
\begin{split}
  \Im(\xi\Phi(w_0))
  &=
  -2\Im\Li_2(e^{u+\varphi(u)})-u\Im\tilde{\varphi}(u)
  \\
  &=
  2\Im\Li_2(e^{-u-\varphi(u)})
  +\Im\bigl(u+\varphi(u)+\pi\sqrt{-1}\bigr)^2
  -u\Im\tilde{\varphi}(u)
  \\
  &=
  2\Im\Li_2(e^{-u-\varphi(u)})
  +u\varphi(u).
\end{split}
\end{equation*}
So we have
\begin{equation*}
\begin{split}
  &\frac{d\,\Im(\xi\Phi(w_0))}{d\,u}
  \\
  =&
  2\Im
  \left(
    \log(1-e^{-u-\varphi(u)})
    \left(
      1
      +
      \sqrt{-1}\frac{d\,\Im\varphi(u)}{d\,u}
    \right)
  \right)
  +\Im\varphi(u)
  +u\frac{d\,\Im\varphi(u)}{d\,u}
  \\
  =&
  \frac{d\,\Im\varphi(u)}{d\,u}
  \left(
    2\log|1-e^{-u-\varphi(u)}|+u
  \right)
  +2\arg(1-e^{-u-\varphi(u)})
  +\Im\varphi(u)
  \\
  =&
  \frac{d\,\Im\varphi(u)}{d\,u}
  \log\Bigl((1-e^{-u-\varphi(u)})(1-e^{-u+\varphi(u)})e^u\Bigr)
  +2\arg(1-e^{-u-\varphi(u)})
  +\Im\varphi(u)
  \\
  =&
  \frac{d\,\Im\varphi(u)}{d\,u}
  \log(e^u+e^{-u}-e^{\varphi(u)}-e^{-\varphi(u)})
  +2\arg(1-e^{-u-\varphi(u)})
  +\Im\varphi(u)
  \\
  =&
  2\arg(1-e^{-u-\varphi(u)})
  +\Im\varphi(u)
  <0
\end{split}
\end{equation*}
since $-\pi/3<\Im\varphi(u)<0$.
\par
Since $\varphi(0)=-\pi\sqrt{-1}/3$ and $\varphi\Bigl(\log\bigl((3+\sqrt{5})/2\bigr)\Bigr)=0$, we have
\begin{equation*}
  \xi\Phi(w_0)
  =
  \begin{cases}
    2\Im\Li_2\left(e^{\pi\sqrt{-1}/3}\right)=1.01494\ldots
    &\quad\text{($u=0$)},
    \\
    2\Im\Li_2\left(\frac{2}{3+\sqrt{5}}\right)=0
    &\quad\text{($u=\log\bigl((3+\sqrt{5})/2\bigr)$)}.
  \end{cases}
\end{equation*}
Therefore $\Im(\xi\Phi(w_0))>0$ for $0<u<\log\bigl((3+\sqrt{5})/2\bigr)$, completing the proof of Lemma~\ref{lem:Phi_0}.
%%%%%%%%%%%%%%%%%%%%%%%%%%%%%%%%%%%%%%%%%%%%%%%%%%%%%%%%%%%%%%%%%%
\bibliography{mrabbrev,hitoshi}

\def\cprime{$'$}
\providecommand{\bysame}{\leavevmode\hbox to3em{\hrulefill}\thinspace}
\providecommand{\MR}{\relax\ifhmode\unskip\space\fi MR }
\providecommand{\MRhref}[2]{%
  \href{http://www.ams.org/mathscinet-getitem?mr=#1}{#2}
}
\providecommand{\href}[2]{#2}
\begin{thebibliography}{10}

\bibitem{Andersen/Hansen:JKNOT2006}
J.~E. Andersen and S.~K. Hansen, \emph{Asymptotics of the quantum invariants
  for surgeries on the figure 8 knot}, J. Knot Theory Ramifications \textbf{15}
  (2006), no.~4, 479--548. \MR{MR2221531}

\bibitem{Dimofte/Gukov:Columbia}
T.~Dimofte and S.~Gukov, \emph{Quantum field theory and the volume conjecture},
  to appear in Contemporary Mathematics "Interactions Between Hyperbolic
  Geometry, Quantum Topology and Number Theory", arXiv:1003.4808, 2010.

\bibitem{Dimofte/Gukov/Lenells/Zagier:CNTP2010}
T.~Dimofte, S.~Gukov, J.~Lenells, and D.~Zagier, \emph{Exact results for
  perturbative {C}hern-{S}imons theory with complex gauge group}, Commun.
  Number Theory Phys. \textbf{3} (2009), no.~2, 363--443. \MR{2551896
  (2010k:58038)}

\bibitem{Dubois:CANMB2006}
J.~Dubois, \emph{Non abelian twisted {R}eidemeister torsion for fibered knots},
  Canad. Math. Bull. \textbf{49} (2006), no.~1, 55--71. \MR{MR2198719}

\bibitem{Dubois/Kashaev:MATHA2007}
J.~Dubois and R.~M. Kashaev, \emph{On the asymptotic expansion of the colored
  jones polynomial for torus knots}, Math. Ann. \textbf{339} (2007), no.~4,
  757--782.

\bibitem{Faddeev:LETMP1995}
L.~D. Faddeev, \emph{Discrete {H}eisenberg-{W}eyl group and modular group},
  Lett. Math. Phys. \textbf{34} (1995), no.~3, 249--254. \MR{1345554
  (96i:46075)}

\bibitem{Gukov:COMMP2005}
S.~Gukov, \emph{Three-dimensional quantum gravity, {C}hern-{S}imons theory, and
  the {A}-polynomial}, Comm. Math. Phys. \textbf{255} (2005), no.~3, 577--627.
  \MR{MR2134725}

\bibitem{Gukov/Murakami:FIC2008}
S.~Gukov and H.~Murakami, \emph{{$SL(2,\mathbf{C})$ Chern-Simons theory and the
  asymptotic behavior of the colored Jones polynomial}}, Modular Forms and
  String Duality (N.~Yui, H.~Verrill, and C.F. Doran, eds.), Fields Inst.
  Comm., vol.~54, Amer. Math. Soc. and Fields Inst., 2008, pp.~261--278.

\bibitem{Hikami/Murakami:Bonn}
K.~Hikami and H.~Murakami, \emph{Representations and the colored jones
  polynomial of a torus knot}, to appear in AMS/IP Studies in Advanced
  Mathematics "Chern-Simons Theory: 20 years after", arXiv:1001.2680, 2010.

\bibitem{Hilden/Lozano/Montesinos:JMASU1995}
H.~Hilden, M.~Lozano, and J.~M. Montesinos-Amilibia, \emph{On a remarkable
  polyhedron geometrizing the figure eight knot cone manifolds}, J. Math. Sci.
  Univ. Tokyo \textbf{2} (1995), no.~3, 501--561. \MR{MR1382519 (96m:57021)}

\bibitem{Jones:BULAM385}
V.~F.~R. Jones, \emph{A polynomial invariant for knots via von {N}eumann
  algebras}, Bull. Amer. Math. Soc. (N.S.) \textbf{12} (1985), no.~1, 103--111.
  \MR{86e:57006}

\bibitem{Kashaev:MODPLA95}
R.~M. Kashaev, \emph{A link invariant from quantum dilogarithm}, Modern Phys.
  Lett. A \textbf{10} (1995), no.~19, 1409--1418. \MR{96j:81060}

\bibitem{Kashaev:LETMP97}
\bysame, \emph{The hyperbolic volume of knots from the quantum dilogarithm},
  Lett. Math. Phys. \textbf{39} (1997), no.~3, 269--275. \MR{98b:57012}

\bibitem{Kirk/Klassen:COMMP93}
P.~Kirk and E.~Klassen, \emph{Chern-{S}imons invariants of $3$-manifolds
  decomposed along tori and the circle bundle over the representation space of
  ${T}\sp 2$}, Comm. Math. Phys. \textbf{153} (1993), no.~3, 521--557.
  \MR{94d:57042}

\bibitem{Klassen:TRAAM1991}
E.~P. Klassen, \emph{Representations of knot groups in {${\rm SU}(2)$}}, Trans.
  Amer. Math. Soc. \textbf{326} (1991), no.~2, 795--828. \MR{MR1008696
  (91k:57003)}

\bibitem{Marsden/Hoffman:Complex_Analysis}
J.~E. Marsden and M.~J. Hoffman, \emph{Basic complex analysis}, W. H. Freeman
  and Company, New York, 1987. \MR{88m:30001}

\bibitem{Masbaum:ALGGT12003}
G.~Masbaum, \emph{Skein-theoretical derivation of some formulas of {H}abiro},
  Algebr. Geom. Topol. \textbf{3} (2003), 537--556 (electronic). \MR{MR1997328
  (2004f:57013)}

\bibitem{Munoz:REVMC2009}
V.~Mu{\~n}oz, \emph{The {${\rm SL}(2,\mathbb C)$}-character varieties of torus
  knots}, Rev. Mat. Complut. \textbf{22} (2009), no.~2, 489--497.
  \MR{MR2553945}

\bibitem{Murakami:ACTMV2008}
H.~Murakami, \emph{An introduction to the volume conjecture and its
  generalizations}, Acta Math. Vietnam. \textbf{33} (2008), no.~3, 219--253.
  \MR{MR2501844}

\bibitem{Murakami:Columbia}
\bysame, \emph{An introduction to the volume conjecture}, to appear in
  Contemporary Mathematics "Interactions Between Hyperbolic Geometry, Quantum
  Topology and Number Theory", arXiv:1002:0126, 2010.

\bibitem{Murakami/Murakami:ACTAM12001}
H.~Murakami and J.~Murakami, \emph{The colored {J}ones polynomials and the
  simplicial volume of a knot}, Acta Math. \textbf{186} (2001), no.~1, 85--104.
  \MR{2002b:57005}

\bibitem{Murakami/Yokota:JREIA2007}
H.~Murakami and Y.~Yokota, \emph{The colored {J}ones polynomials of the
  figure-eight knot and its {D}ehn surgery spaces}, J. Reine Angew. Math.
  \textbf{607} (2007), 47--68. \MR{MR2338120}

\bibitem{Neumann/Zagier:TOPOL85}
W.~D. Neumann and D.~Zagier, \emph{Volumes of hyperbolic three-manifolds},
  Topology \textbf{24} (1985), no.~3, 307--332. \MR{87j:57008}

\bibitem{Porti:MAMCAU1997}
J.~Porti, \emph{Torsion de {R}eidemeister pour les vari\'et\'es hyperboliques},
  Mem. Amer. Math. Soc. \textbf{128} (1997), no.~612, x+139. \MR{MR1396960
  (98g:57034)}

\bibitem{Riley:QUAJM31984}
R.~Riley, \emph{Nonabelian representations of {$2$}-bridge knot groups}, Quart.
  J. Math. Oxford Ser. (2) \textbf{35} (1984), no.~138, 191--208. \MR{MR745421
  (85i:20043)}

\bibitem{Thurston:GT3M}
W.~P. Thurston, \emph{{The Geometry and Topology of Three-Manifolds}},
  Electronic version 1.1 - March 2002,
  http://www.msri.org/publications/books/gt3m/.

\bibitem{Yoshida:INVEM85}
T.~Yoshida, \emph{The $\eta$-invariant of hyperbolic $3$-manifolds}, Invent.
  Math. \textbf{81} (1985), no.~3, 473--514. \MR{87f:58153}

\end{thebibliography}
\bibliographystyle{amsplain}
\end{document}